\documentclass[11pt,twoside]{article}
\usepackage{latexsym}
\usepackage{amssymb,amsbsy,amsmath,amsfonts,amssymb,amscd,amsthm}
\usepackage{epsfig, graphicx}
\usepackage{color}
\usepackage{enumerate}
\usepackage{pdfsync}
\newcommand{\norm}[1]{\|#1 \|}
\newcommand{\R}{\mathbb{R}}
\newcommand{\N}{\mathbb{N}}

\newcommand {\e}  {\varepsilon}

\newcommand{\caT}{{\cal T}}

\newcommand {\f}   {\frac}

\newcommand{\beq}{\begin{equation}}
\newcommand{\beqa}{\begin{eqnarray}}
\newcommand{\bea} {\begin{array}{ll}}
\newcommand{\beqan}{\begin{eqnarray*}}
\newcommand{\eeq}{\end{equation}}
\newcommand{\eeqa}{\end{eqnarray}}
\newcommand{\eeqan}{\end{eqnarray*}}
\newcommand{\eea} {\end{array}}
\newtheorem{theorem}{Theorem}[section]
\newtheorem{lemma}[theorem]{Lemma}

\newtheorem{proposition}[theorem]{Proposition}
\newtheorem{corollary}[theorem]{Corollary}
\newtheoremstyle{remarkb}
	{}
	{}
	{\normalfont}
	{}
	{\bfseries}
	{}
	{ }
	{}
\theoremstyle{remarkb}
\newtheorem{remark}[theorem]{Remark}
\numberwithin{equation}{section}
\newenvironment{acknowledgment}{\noindent{\bf Acknowledgment}}{}
\title{\Large Existence, uniqueness and asymptotic behavior of the solutions to the fully parabolic Keller-Segel system in the plane}

\author{Lucilla Corrias~$^{\text{a}}$, Miguel Escobedo~$^{\text{b}}$, Julia Matos~$^{\text{a}}$}
\date{\today}
\begin{document}
\maketitle
\pagestyle{plain}
\pagenumbering{arabic}

\begin{abstract}
In the present article we consider several issues concerning the doubly parabolic Keller-Segel system \eqref{KSu}-\eqref{KSv} in the plane, when the initial data belong to critical scaling-invariant Lebesgue spaces. More specifically, we analyze the global existence of integral solutions, their optimal time decay, uniqueness and positivity, together with the uniqueness of self-similar solutions. In particular, we prove that there exist integral solutions of any mass, provided that $\e>0$ is sufficiently large. With those results at hand, we are then able to study the large time behavior of global solutions and prove that in the absence of the degradation term $(\alpha=0)$ the solutions behave like self-similar solutions, while in presence of the degradation term $(\alpha>0)$ global solutions behave like the heat kernel.
\end{abstract}


\section{Introduction}
This paper is devoted to the analysis of the parabolic-parabolic Keller-Segel system
\begin{eqnarray}
\label{KSu}
u_t&=&\Delta u-\nabla\cdot(u\,\nabla v)\,,\\
\label{KSv}
\e\,v_t&=&\Delta v+u-\alpha\, v\,,
\end{eqnarray}
in the whole plane $\R^2$, where $\e>0$, while $\alpha\ge0$. 

There exists a huge mathematical literature on system \eqref{KSu}-\eqref{KSv} in any space dimension. A particular interest is addressed to  the case of dimension two, generally considered as the natural one from the point of view of the biological interpretation of the model. In that case, most of the existing  results concern the parabolic-elliptic Keller-Segel system ($\e=0$). The goal of this paper is to analyse \eqref{KSu}-\eqref{KSv} for arbitrary positive values of $\e$. As we will see, this parameter is important not  only to determine whether we are in the doubly parabolic or in  the parabolic-elliptic case. It also represents different diffusivities on  $u$ and $v$ and that will be important for the existence, uniqueness and long time behavior properties of solutions.
 
In the remaining  of this introduction, we briefly describe our results and present them in the context of what is previously known. For the sake of clearness, due to the vast literature existing on the Keller-Segel system, we shall only mention papers that study the two dimensional case.  

Let us  recall first that a formal integration of the equation (\ref{KSu}) with respect to $x$ over all of  $\R^2$  indicates that the integral of  $u(t)$ is constant in time:
\[
M:=\int_{\R^2}u(x,t)dx=\int_{\R^2}u_0(x) dx\,,\quad t>0\,.
\]
This property will be proved to be true, for at least some of the solutions. On the other hand,  when $\alpha=0$, system \eqref{KSu}-\eqref{KSv} is invariant under the following space-time scaling
\beq
u_\lambda(x,t)=\lambda^2u(\lambda x,\lambda^2t)\,,\quad v_\lambda(x,t)=v(\lambda x,\lambda^2t)\,,\quad\lambda>0\,,
\label{eq:scaling}
\eeq
that preserves the  integral of $u(t)$ on $\R^2$. Scaling (\ref{eq:scaling}) also preserves the $L^2(\R^2)$ norm of $|\nabla v(t)|$. 
Hence, the space of functions $(u, v) \in L^1(\R^2)\times  \dot{H}^1(\R^2)$ arises very naturally, 
where $\dot{H}^1(\R^2)$ denotes the homogeneous Sobolev space defined via Fourier transform as the completion of $C^\infty_0(\R^2)$ under the seminorm
$\|v\|_{\dot{H}^1(\R^2)}^2=\int_{\R^2}|\xi|^2|\hat{v}(\xi)|^2\,d\xi$.

Moreover, the conserved mass $M$ of $u(t)$ should play an important role in the analysis of \eqref{KSu}-\eqref{KSv}. This is  the case for the two dimensional parabolic-elliptic system, that shows the well known threshold phenomenon \cite{BDP} : positive solutions are global in time if the mass $M$ is below $8\pi$ and blow-up in finite time if the mass is above $8\pi$. The critical case $M=8\pi$ has been studied in \cite{BCM}, where the authors show that positive solutions aggregate as $t\to\infty$ (see also \cite{Miz}, and \cite{BKLN} for the radially symmetric case). The global existence result for the mass of $u$ below $8\pi$ has been extended to the two dimensional parabolic-parabolic system in \cite{CC08,Miz}. We prove here that when $\e>0$, global solutions  may exist, even with large mass $M$.

In all the articles that are mentioned above, the authors consider positive solutions of weak type and the key tool used to obtain the necessary  a priori estimates for the global existence result is the free energy naturally associated to \eqref{KSu}-\eqref{KSv}, i.e.
\begin{equation}
\mathcal E(t) := \int_{\R^2} u\log u\,dx - \int_{\R^2} u\,v\,dx + \f12\int_{\R^2} |\nabla v|^2\,dx + \f\alpha2 \int_{\R^2} v^2\,dx\,.
\label{eq:fe}
\end{equation}

These weak solutions also satisfy the expected parabolic regularizing effect. However, this regularizing phenomenon is not proved to be uniform in time (\cite{CC08}). In order to overcome this problem and obtain the optimal decay in time estimates, we consider here the solutions of \eqref{KSu}-\eqref{KSv} in the following integral sense: 
\beq\label{eq:Duhamel_u}
u(t)=G(t)* u_0-\sum_{i}\int_0^t \partial_{i} G(t-s)*(u(s)\partial_{i} v(s))\ ds\,,
\eeq
\beq\label{eq:Duhamel_v}
v(t)=e^{-(\alpha/\e)\,t}\, G(\e^{-1}t)* v_0 +\e^{-1}\int_0^t e^{-(\alpha/\e)(t-s)} G(\e^{-1}(t-s))* u(s) \ ds\,.
\eeq
where $G(x,t)=\f1{4\pi t}\,e^{-|x|^2/4t}$ is the heat kernel.

These integral solutions are very natural and have been studied by several authors (see \cite{B98, BiBr, BGK, Ferreira,NSU03, N06} and Remark \ref{Rk:global existence}). In the present article, we prove the global  existence of solutions for initial  data $(u_0, v_0)\in  L^1(\R^2)\times \dot{H}^1(\R^2)$ under some condition that involves the size of the initial data and $\e$ (see Theorem~\ref{th:existence critical}). We then obtain the regularizing effects typical of the parabolic problems, i.e. the optimal time decay rates of~$\|u(t)\|_p$, $\|\nabla u(t)\|_p$ for $p\ge 1$, and $\|\nabla v(t)\|_r$, $\|\Delta v(t)\|_r$ for $r\ge 2$ (see Proposition~\ref{pr:moreestimates}). In particular, we obtain the uniform in time boundedness of~$u(t)$, without requiring the boundedness of the initial data (see also \cite{BGK} for the case $v_0=0$). These decay rates are then used  for the analysis of the long time behavior of the solutions. With these estimates at hand, we also  prove the continuous dependence of the global integral solutions with respect to the initial data. As a consequence, we deduce  the uniqueness and the positivity of the solution itself (see Theorem~\ref{th:continuous dependence} and Corollary \ref{crl:uniqueness positivity}). To the best of our knowledge, the contraction property for the distance between two solutions of \eqref{KSu}-\eqref{KSv} was previously proved in \cite{BKLN} for the parabolic-elliptic radially symmetric case, and in \cite{CLM}, in the context of the gradient flow formulation of \eqref{KSu}-\eqref{KSv}, for initial data $u_0\in (L^1\cap L^\infty)(\R^2)$ with finite second moment and $v_0\in H^1(\R^2)$ (see also the asymptotic stability result in \cite{Ferreira}). 

Our second result is about the uniqueness of positive integrable and rapidly decaying self-similar solutions of \eqref{KSu}-\eqref{KSv} with $\alpha=0$. These solutions are invariant with respect to the scaling \eqref{eq:scaling} and therefore provide a uniparametric family $(u_M, v_M)$ indexed by the mass $M$. The existence of such family has been considered by several authors (see \cite{B95,B98,BCD,BKLN, MMY99, NSY02,Y01} and  references therein). In \cite{Ferreira,N06} and for $\e=1$, it has been proved the existence and uniqueness of small self-similar solutions with small initial data, through the analysis of the integral formulation of \eqref{KSu}-\eqref{KSv} (see Remark \ref{rk:comparison}). However, the question of uniqueness in general is still largely open. Analyzing directly the profiles of $(u_M, v_M)$, we show in Theorem~\ref{th:uniqueness_self-similar} that for any $\e>0$, the positive integrable and rapidly decaying self-similar solution $(u_M, v_M)$ with $M$ less than some positive constant $\widetilde M(\e)\in[4\pi, 8\pi]$, that only depends on $\e$, is unique (see Figure \ref{fig:Mtilde}). Moreover, for all  $\e\in (0, 1/2]$,  $\widetilde M(\e)=8\pi$. Hence, in that case, for every $M<8\pi $ the self-similar solution $(u_M, v_M)$ is unique, exactly as for the parabolic-elliptic case~\cite{BKLN}.

The third result of this paper concerns the long time behavior of the global integral solutions of \eqref{KSu}-\eqref{KSv}. Due to the scaling invariance of the system in absence of the degradation term for $v$, it is natural to expect that, if $\alpha=0$, global solutions behave asymptotically in time as self-similar solutions of the same system. This is indeed observed in the case of the non-linear heat equation \cite{EK87b} and of a convection-diffusion equation \cite{EZ91a}. This is also the case for the two dimensional parabolic-elliptic Keller-Segel system with $\alpha=0$ and $M\le8\pi$ (see \cite{BKLN,BCM,BDP,BDEF,CD,EM}). The case of the doubly parabolic Keller-Segel system with $\e=1$, has been studied in \cite{Ferreira,N06}. In particular, the authors in \cite{N06} prove that the long time asymptotic behaviour of the integral solution $u$ is given by the self-similar solution $u_M$ in the $L^p(\R^2)$ space, $p\in(4/3,2)$, if $(1+|x|^2) u_0\in L^1(\R^2)$, $|\nabla v_0|\in(L^1\cap L^2)(\R^2)$ and $M$ is sufficiently small. In \cite{Ferreira} the authors prove that each self-similar solution furnish an attractor-basin for the global integral solution issued by a smooth perturbation of the initial data of the self-similar solution itself (see Remark \ref{S4Rem2}).

We prove in Theorem \ref{th:longtime-alpha=0} that if $\e>0$ and $(u,v)$ is a non-negative global solution of \eqref{KSu}-\eqref{KSv} satisfying the optimal in  time decay rates and such that the mass $M$ is below the same threshold $\widetilde M(\e)$ assuring the uniqueness of the self-similar solution $(u_M,v_M)$, then 
\[
t^{(1-1/p)}||u(t)-u_M(t)|| _{ L^p(\R^2) }+t^{1/2-1/r}||\nabla v(t)-\nabla v_M(t)|| _{ L^r(\R^2) }\to 0\,,\quad\text{as } t\to \infty\,,
\] 
for all $p\in [1, \infty]$ and $r\in [2, \infty]$. 
Therefore, in the case of $0<\e\leq \frac{1}{2}$, a global non-negative solution $(u,v)$ of 
\eqref{KSu}-\eqref{KSv} has the same long time behavior than the unique self-similar $(u_M,v_M)$, provided
$M<8\pi$.

For the seek of completeness, we also consider the case $\alpha>0$ and $\e>0$.
We prove then that the long time behavior of global integral solutions is the same as that of the heat kernel (see Theorem \ref{th:asymptbeh-alphapos}). In that case, the positivity of the initial data is not required. 

The paper is organized as follows. In Section \ref{sec:existence}, we give the local and global existence result of integral solutions. Section \ref{sec:self-similar} is devoted to the uniqueness issue of forward self-similar solutions. In Section \ref{sec:alpha0} we analyze the long time behavior of integral solution in the case $\alpha=0$, while the case $\alpha>0$ is considered in Section \ref{sec:alphapos}.
\section{Existence of integral solutions and decay estimates}
\label{sec:existence}
Our first result concerns the global existence of the integral solutions \eqref{eq:Duhamel_u}-\eqref{eq:Duhamel_v} and their optimal time decay rates, the same that for the linear heat equation. It is obtained using a fixed point type argument in an ad hoc complete metric spaces, a classical and efficient technique that gives the desired optimal time decay in counterpart. Moreover, the condition on the initial data, necessary for the global existence of the corresponding solution, depends on $\e$ in such a way that each mass $M$ may leads to a global-in-time-solution (see Remark \ref{rm:small e}).

\begin{theorem}[Local and global existence]\label{th:existence critical}
Let $\e>0$, $\alpha\ge0$, $u_0\in L^1(\R^2)$ and $v_0\in \dot{H}^1(\R^2)$.
There exist $\delta=\delta(\|u_0\|_{L^1(\R^2)},\e)>~0$ and $T=T(\|u_0\|_{L^1(\R^2)},\e)>0$ such that if $\|\nabla v_0\|_{L^2(\R^2)}<\delta$ there exist an integral solution $(u,v)$ of \eqref{KSu}-\eqref{KSv} with $u\in L^\infty((0,T);L^1(\R^2))$ and $|\nabla v|\in L^\infty((0,T);L^2(\R^2))$. Moreover, the total mass $M$ is conserved and there exists a constant $C=C(\e)$ such that  if $\|u_0\|_{L^1(\R^2)}<C(\e)$, the solution is global and
\beq\label{reg_effetct1}
t^{(1-\f1p)}\norm{u(t)}_{L^p(\R^2)}\le C(\|u_0\|_{L^1(\R^2)},\e)\,,\quad t>0\,,
\eeq
\beq\label{reg_effetct2}
t^{(\f12-\f1r)}\|\nabla v(t)\|_{L^r(\R^2)}\le C(\|u_0\|_{L^1(\R^2)},\e)\,,\quad t>0\,,
\eeq
for all $p\in[1,\infty]$ and $r\in[2,\infty]$. 
\end{theorem}
\begin{proof}
We shall prove the theorem in several steps. 
The classical regularizing effect of the heat kernel will be also employed in all of these steps as well as the notation below for the beta function
\[
B_{(x,y)}:=\int_0^1\sigma^{-x}(1-\sigma)^{-y}\,d\sigma\,,\qquad x,\,y\in(0,1)\,.
\]

{\sl First step : local existence.} For $p\in(2,4)$ arbitrarily fixed, $T>0$ and $\eta>0$ to be chosen later, let us define $E_p:= L^\infty((0,T);L^{1}(\R^2))\cap L^\infty_{loc}((0,T);L^p(\R^2))$ and 
\[
X_p:=\{u\in E_p\ :\ \norm{u(t)}_{L^{1}(\R^2)}\le A+1,\quad t^{(1-\f1p)}\norm{u(t)}_{L^p(\R^2)}\le\e^{(1-\f1p)}\eta,\ t\in(0,T)\}\,,
\]
where $A:=\|u_0\|_{L^1(\R^2)}$. Then, $(X_p, d_p)$ with the distance $d_p(u_1, u_2)$ defined as following
\[
d_p(u_1,u_2):=\e^{-(1-\f1p)}\sup_{0<t<T}t^{(1-\f1p)}\norm{u_1(t)-u_2(t)}_{L^p(\R^2)}\,,
\]
is a nonempty complete metric space. Next, for $u_0$ and $v_0$ given as in the statement of the theorem and for a fixed $u\in X_p$, we define $v$ as in  \eqref{eq:Duhamel_v}~and
\beq\label{eq:Tau(u)}
\caT(u)(t):=G(t)* u_0-\sum_{i}\int_0^t \partial_{i} G(t-s)*(u(s)\partial_{i} v(s))\ ds\,.
\eeq
The estimate of $\|\nabla v(t)\|_{L^r(\R^2)}$ from \eqref{eq:Duhamel_v} is crucial and given, for all $r\ge p$, by
\beq
\begin{split}
\|\nabla v(t)\|_{L^r(\R^2)}&\le C_0(r)\e^{(\f12-\f1r)}\,t^{-(\f12-\f1r)}\|\nabla v_0\|_{L^2(\R^2)}+\e^{-1}\,C_1(p,r)\int_0^t\f{\e^{\f1p-\f1r+\f12}}{(t-s)^{\f1p-\f1r+\f12}}\|u(s)\|_{L^p(\R^2)}\, ds\\
&\le\left[ C_0(r)t^{-(\f12-\f1r)}\|\nabla v_0\|_{L^2(\R^2)}+C_1(p,r)\,\eta\int_0^t\f1{(t-s)^{\f1p-\f1r+\f12}}\,\f1{s^{1-\f1p}}\,ds\right]\e^{(\f12-\f1r)}\\
&=\left[C_0(r)\|\nabla v_0\|_{L^2(\R^2)}+C_1(p,r)\,B_{(1-\f1p,\f1p-\f1r+\f12)}\,\eta\right]\e^{(\f12-\f1r)}\,t^{-(\f12-\f1r)}\,.
\end{split}
\label{est:grad v}
\eeq
This establishes \eqref{reg_effetct2} for $r\in[p,\infty]$ locally in time, after choosing $\eta$. In particular, for $r=\infty$, it holds
\beq\label{est:grad v infty}
\|\nabla v(t)\|_{L^\infty(\R^2)}\le\left[(8\pi)^{-\f12}\|\nabla v_0\|_{L^2(\R^2)}+C_1(p,\infty)\,B_{(1-\f1p,\f1p+\f12)}\,\eta\right]\e^{\f12}\,t^{-\f12}\,.
\eeq
Therefore, from \eqref{eq:Tau(u)} and \eqref{est:grad v infty}, we obtain
\beq\label{est:T(u)0}
\begin{split}
\|\caT(u)(t)\|_{L^{1}(\R^2)}&\le A+2\sqrt\pi(A+1)\int_0^t\f1{(t-s)^{\f12}}\,\|\nabla v(s)\|_{L^\infty(\R^2)}\,ds\\
&\le A+2\sqrt\pi(A+1)B_{(\f12,\f12)}\left[(8\pi)^{-\f12}\|\nabla v_0\|_{L^2(\R^2)}+C_1(p,\infty)\,B_{(1-\f1p,\f1p+\f12)}\eta\right]\e^{\f12}\\
&\le A+1
\end{split}
\eeq
provided
\beq\label{est:T(u)1}
(8\pi)^{-\f12}\|\nabla v_0\|_{L^2(\R^2)}+C_1(p,\infty)\,B_{(1-\f1p,\f1p+\f12)}\eta\le\left(2\sqrt\pi(A+1)B_{(\f12,\f12)}\e^{\f12}\right)^{-1}\,.
\eeq
Similarly, using \eqref{est:grad v} for $q$ fixed such that $\f1p\ge\f1q>\f12-\f1p$, it holds
\beq
\begin{split}
 t^{(1-\f 1{p})}&\|\caT(u)(t)\|_{L^p(\R^2)}\le  t^{(1-\f1{p})}\|G(t)*u_0\|_{L^p(\R^2)}\\
&\quad+2\,t^{(1-\f 1{p})}C_2(q)\int_0^t\f1{(t-s)^{\f 1q+\f12}}\,\|u(s)\|_{L^p(\R^2)}\|\nabla v(s)\|_{L^q(\R^2)}\,ds\\
&\le t^{(1-\f1p)}\|G(t)*u_0\|_{L^p(\R^2)}\\
&\quad+2\,\e^{(1-\f1p)}\eta\,C_2(q)B_{(\f32-\f1p-\f1q,\f1q+\f12)}\left[C_0(q)\|\nabla v_0\|_{L^2(\R^2)}+C_1(p,q)\,B_{(1-\f1p,\f1p-\f1q+\f12)}\eta\right]\e^{(\f12-\f1q)}\\
&\le t^{(1-\f1p)}\|G(t)*u_0\|_{L^p(\R^2)}+\f12\,\e^{(1-\f1p)}\eta
\end{split}
\label{est:T(u)2}
\eeq
provided
\beq
C_0(q)\|\nabla v_0\|_{L^2(\R^2)}+C_1(p,q)\,B_{(1-\f1p,\f1p-\f1q+\f12)}\,\eta\le
\left(4\,C_2(q)B_{(\f32-\f1p-\f1q,\f1q+\f12)}\e^{(\f12-\f1q)}\right)^{-1}\,.
\label{est:T(u)3}
\eeq
Furthermore, since $\lim_{t\to0}t^{(1-\f1p)}\|G(t)*u_0\|_{L^p(\R^2)}=0$, (see \cite{BC96}), after choosing $\eta$, we can take  $T>0$ such that  for $t\in[0,T]$ it holds 
\beq\label{def of T}
t^{(1-\f1p)}\|G(t)*u_0\|_{L^p(\R^2)}\le\f12\e^{(1-\f1p)}\eta\,.
\eeq

Next, taking $u_1,u_2\in X_p$, we have exactly as in \eqref{est:grad v}, for all $r\ge p$,
\[
\begin{split}
\|\nabla v_1(t)-\nabla v_2(t)\|_{L^r(\R^2)}&\le \e^{-1}C_1(p,r)\int_0^t\f{\e^{\f1p-\f1r+\f12}}{(t-s)^{\f1p-\f1r+\f12}}\|u_1(s)-u_2(s)\|_{L^p(\R^2)}\, ds\\
&\le C_1(p,r)\,B_{(1-\f1p,\f1p-\f1r+\f12)}\,\,d_p(u_1,u_2)\,\e^{(\f12-\f1r)}\,t^{-(\f12-\f1r)}\,,
\end{split}
\label{est:T(u)4}
\]
and exactly as in \eqref{est:T(u)2},  for $q$ fixed such that $\f1p\ge\f1q>\f12-\f1p$ ,
\beq
\begin{split}
t^{(1-\f1p)}&\|\caT(u_1)(t)-\caT(u_2)(t)\|_{L^p(\R^2)}\le 2\,t^{(1-\f1p)}C_2(q)\int_0^t\f1{{(t-s)}^{\f1q+\f12}}\,\|u_1(s)-u_2(s)\|_{L^p(\R^2)}\,\|\nabla v_1(s)\|_{L^q(\R^2)}\,ds\\
&\qquad+2\,t^{(1-\f1p)}C_2(q)\int_0^t\f1{(t-s)^{\f1q+\f12}}\,\|u_2(s)\|_{L^p(\R^2)}\,\|\nabla v_1(s)-\nabla v_2(s)\|_{L^q(\R^2)}\,ds\\
&\le 2\,C_2(q)\,\e^{(1-\f1p)}d_p(u_1,u_2)B_{(\f32-\f1p-\f1q,\f1q+\f12)}
\left[C_0(q)\|\nabla v_0\|_{L^2(\R^2)}+C_1(p,q)\,B_{(1-\f1p,\f1p-\f1q+\f12)}\,\eta\right]\e^{(\f12-\f1q)}\\
&\qquad+2\,C_2(q)\,\e^{(1-\f1p)}\,\eta\,B_{(\f32-\f1p-\f1q,\f1q+\f12)}\left[ C_1(p,q)\,B_{(1-\f1p,\f1p-\f1q+\f12)}\,\,d_p(u_1,u_2)\,\e^{(\f12-\f1q)}\right]\\
&=2\,C_2(q)\,\e^{(1-\f1p)}d_p(u_1,u_2)B_{(\f32-\f1p-\f1q,\f1q+\f12)}\left[C_0(q)\|\nabla v_0\|_{L^2(\R^2)}+2\,C_1(p,q)\,B_{(1-\f1p,\f1p-\f1q+\f12)}\,\eta\right]
\e^{(\f12-\f1q)}\\
&\le\e^{(1-\f1p)}d_p(u_1,u_2)
\end{split}
\label{est:T(u)5}
\eeq
provided
\beq\label{est:T(u)6}
C_0(q)\|\nabla v_0\|_{L^2(\R^2)}+2\,C_1(p,q)\,B_{(1-\f1p,\f1p-\f1q+\f12)}\,\eta\le\left(2\,C_2(q)\,B_{(\f32-\f1p-\f1q,\f1q+\f12)}\e^{(\f12-\f1q)}\right)^{-1}\,.
\eeq

To conclude, from \eqref{est:T(u)1}, \eqref{est:T(u)3} and \eqref{est:T(u)6}, we choose $\delta>0$ and $\eta>0$ such that if $\|\nabla v_0\|_{L^2(\R^2)}<\delta$, inequalities \eqref{est:T(u)0}, \eqref{est:T(u)2} and \eqref{est:T(u)5} are satisfied. Then, we choose $T$ such that \eqref{def of T} is also satisfied. Consequently, $\caT$ is a contraction from $X_p$ to $X_p$. The local existence of an integral solution follows applying the Banach fixed point Theorem. It is worth noticing that the choice of $\delta$, $\eta$ and $T$ depend on $\e$, $A$ and the previously fixed $p$ and $q$.

{\sl Second step : regularizing effects.} Let $p$, $\eta$ and $T$ be the same fixed in the previous step and  let $q\in(p,\infty)$. Using \eqref{est:grad v} with $r\ge p$ such that $\f12-\f1p<\f1r<\f12-\f1p+\f1q$, and the fact that $u\in X_p$, it holds for $t\in(0,T)$
\[\label{est:reg_effetcs}
\begin{split}
 t^{(1-\f 1q)}\|u(t)\|_{L^q(\R^2)}&\le  t^{(1-\f1q)}\|G(t)*u_0\|_{L^q(\R^2)}\\
 &\ +C\,t^{(1-\f 1q)}\int_0^t\f1{(t-s)^{\f1p+\f 1r-\f1q+\f12}}\,\|u(s)\|_{L^p(\R^2)}\|\nabla v(s)\|_{L^r(\R^2)}\,ds\le C(\e,A)\,.
\end{split}
\]
Therefore, \eqref{reg_effetct1} is established up to now for $q\in[p,\infty)$. For $q\in(1,p)$, \eqref{reg_effetct1} follows by interpolation. For $q=\infty$, taking the $L^\infty$ norm of the identity 
\[
u(2t)=G(t)*u(t)-\sum_{i}\int_0^t \partial_{i} G(t-s)*(u(s+t)\partial_{i} v(s+t))\ ds\,,
\]
where $2\,t\in(0,T)$, and using \eqref{est:grad v infty}, we obtain
\[
\|u(2t)\|_{L^\infty(\R^2)}\le \f Ct(A+1)+C(\e,A)\int_0^t\f1{(t-s)^{\f1p+\f12}}\f1{(s+t)^{1-\f1p}}\f1{(s+t)^{\f12}}\,ds\le C(\e,A)\,t^{-1}\,.
\]
Finally, \eqref{reg_effetct2} has been established in the previous step for $r\in[p,\infty]$. For $r\in[2,p)$, it follows easily by \eqref{reg_effetct1}.

{\sl Third step : global existence.} Let now $p>1$ be arbitrarily fixed. The identity \eqref{eq:Duhamel_u} satisfied by the solution $u$ implies that the function $f_p(t):=\sup_{s\in(0,t)}s^{(1-\f1p)}\norm{u(s)}_{L^p(\R^2)}$ satisfies for $t\in(0,T]$
\beq
\begin{split}
f_p(t)&\le C_3(p)\,A+2\,C_2(r)t^{(1-\f1p)}\,f_p(t)\int_0^t\f1{(t-s)^{\f 1{r}+\f12}}\,\f1{s^{1-\f1p}}\|\nabla v(s)\|_{L^r(\R^2)}\,ds\\
&\le C_3(p)\,A\\
&\quad+2\,C_2(r)\,f_p(t)\,B_{(\f32-\f1p-\f1r,\f1r+\f12)}
\left[C_0(r)\|\nabla v_0\|_{L^2(\R^2)}+C_1(p,r)\,\e^{(\f1p-1)} B_{(1-\f1p,\f1p-\f1r+\f12)}\,f_p(t)\right]\e^{(\f12-\f1r)}\,.
\end{split}
\label{est:fp}
\eeq
Here, we have estimate $\|\nabla v(s)\|_{L^r(\R^2)}$ as in \eqref{est:grad v}, chosen an appropriate $r>2$ (with respect to the fixed $p$) and take into account the increasing behavior of $f_p(t)$. Therefore, rearranging the terms in \eqref{est:fp} and renoting some constants for simplicity, it holds
\[
\e^{(\f1p-\f1r-\f12)}K_1(p,r)\,f_p^2(t)+[\e^{(\f12-\f1r)}K_2\|\nabla v_0\|_{L^2(\R^2)}-1]f_p(t)+C_3(p)\,A\ge0\,.
\]
Finally, since $\lim_{t\to0}f_p(t)=0$,  $f_p(t)$ stay upper bounded whenever
\beq
\e^{(\f12-\f1r)}K_2(p,r)\|\nabla v_0\|_{L^2(\R^2)}<1
\label{eq:smallness_conditions1}
\eeq
and
\beq
[\e^{(\f12-\f1r)}K_2(p,r)\|\nabla v_0\|_{L^2(\R^2)}-1]^2-4\,\e^{(\f1p-\f1r-\f12)}K_1(p,r)\,C_3(p)\,A>0\,.
\label{eq:smallness_conditions2}
\eeq
Noticing that conditions \eqref{eq:smallness_conditions1} and \eqref{eq:smallness_conditions2} are equivalent to
\beq
\left(4\,\e^{(\f1p-\f1r-\f12)}K_1(p,r)\,C_3(p)\,A\right)^{\f12}+\e^{(\f12-\f1r)}K_2(p,r)\|\nabla v_0\|_{L^2(\R^2)}<1\,,
\label{eq:smallness_conditions3}
\eeq
the global existence of the solution follows under the smallness condition \eqref{eq:smallness_conditions3}. 
\end{proof}

\begin{remark}\label{rm:small e}
It is worth noticing that all the conditions on $\|\nabla v_0\|_{L^2(\R^2)}$ established in the previous theorem, vanishes as $\e\to0$. On the other hand, since  we  necessarily have $\f1p-\f1r-\f12<0$ (cf \eqref{est:grad v}),  the smaller is $\varepsilon$ the more restrictive is the condition \eqref{eq:smallness_conditions3} that is required on $A$ in order to have a global solution.  However, for the same reason, the larger $\e$ becomes, the larger may the constant $A$ be chosen. Therefore the doubly parabolic system has solutions for initial data $(u_0, v_0)\in L^1(\R^2)\times \dot{H}^1(\R^2)$ with the mass as large as we like, whenever $\varepsilon $ is sufficiently large. A similar result has been proved in~\cite{BGK} for $v_0=0$ and $u_0$ a finite Radon measure on $\R^2$. 
\end{remark}

Next, we improve the previous theorem showing the optimal time decay of $\nabla u(t)$ and $\Delta v(t)$, for which we need the variant below of the Gronwall's lemma.

\begin{lemma}[\cite {BC96}]\label{lem:Gronwall}
Let $T>0$, $A\ge0$, $\alpha,\beta\in[0,1)$ and let $f$ be a nonnegative function with $f\in L^p(0,T)$ for some $p>1$ such that $p'\max\{\alpha,\beta\}<1$. Then, if $\phi\in L^\infty(0,T)$ satisfies
\[
\phi(t)\le A\,t^{-\alpha}+\int_0^t(t-s)^{-\beta}\,f(s)\,\phi(s)\,ds\,,\qquad a.e.\ t\in(0,T]\,,
\]
there exists $C=C(T,\alpha,\beta,p,\|f\|_{L^p(0,T)})$ such that 
\[
\phi(t)\le A\,C\,t^{-\alpha}\,,\qquad a.e.\ t\in(0,T]\,.
\]
\end{lemma}
\begin{proposition}\label{pr:moreestimates}
The global integral solution $(u,v)$ of \eqref{KSu}-\eqref{KSv} given by Theorem \ref{th:existence critical} satisfies
\beq
\|\nabla u(t)\|_{L^p(\R^2)}\le C\,t^{-(1-\f1p)-\f12}\,,\qquad t>0\,,
\label{eq:grad u}
\eeq
\beq
\|\Delta v(t)\|_{L^r(\R^2)}\le C\,t^{-(\f12-\f1r)-\f12}\,,\qquad t>0\,,
\label{eq:delta v}
\eeq
for all $p\in[1,\infty]$ and $r\in[2,\infty]$,  where $C=C(\|u_{0}\|_{L^1(\R^2)},\e)>0$.
\end{proposition}
\begin{proof}
We shall make use of the rescaled solution $(u_\lambda,v_\lambda)$ defined in \eqref{eq:scaling} and of the regularizing effects \eqref{reg_effetct1} and \eqref{reg_effetct2}, giving respectively the estimates below, with constants $C$ independent of $\lambda$, 
\beq
\|u_\lambda(t)\|_{L^p(\R^2)}=\lambda^{2-\f2p}\|u(\lambda^2t)\|_{L^p(\R^2)}\le C\,t^{-(1-\f1p)}\,,\quad t>0\,,\ p\in[1,\infty]\,,
\label{reg_effetct1_lambda}
\eeq
and
\beq
\|\nabla v_\lambda(t)\|_{L^r(\R^2)}=\lambda^{1-\f2r}\|\nabla v(\lambda^2t)\|_{L^r(\R^2)}\le C\,t^{-(\f12-\f1r)}\,,\quad t>0\,,\ r\in[2,\infty]\,.
\label{reg_effetct2_lambda}
\eeq

Assume $p>2$ and let $t>0$ and $\tau>0$ be arbitrarily fixed. Then, taking the $L^\infty$ norm of the identity
\beq\label{eq-delta v lambda}
\Delta v_\lambda(t+\tau)=e^{-(\alpha/\e)\,t}\,\sum_i\partial_i G(\e^{-1}t)*\partial_iv_\lambda(\tau) +\e^{-1}\sum_i\int_0^t e^{-(\alpha/\e)(t-s)}\partial_i G(\e^{-1}(t-s))*\partial_i u_\lambda(s+\tau)ds\,,
\eeq
and using \eqref{reg_effetct2_lambda} with $r=\infty$, we get 
\[
\|\Delta v_\lambda(t+\tau)\|_{L^\infty(\R^2)}\le C\,t^{-\f 12}\tau^{-\f 12}+C\int_0^t\f1{(t-s)^{\f1p+\f12}}\|\nabla u_\lambda(s+\tau)\|_{L^p(\R^2)}\,ds\,.
\]
On the other hand, taking the $L^p$ norm of
\[
\nabla u_\lambda(t+\tau)=\nabla G(t)* u_\lambda(\tau)-\sum_{i}\int_0^t \partial_{i} G(t-s)*\nabla (u_\lambda(s+\tau)\partial_{i} v_\lambda(s+\tau))\ ds\,,
\]
and using \eqref{reg_effetct1_lambda} and \eqref{reg_effetct2_lambda} again, we obtain
\beq
\begin{split}
\|\nabla u_\lambda(t+\tau)\|_{L^p(\R^2)}&\le C\,t^{-\f 12}\tau^{-(1-\f 1p)}+C\int_0^t\f1{(t-s)^{\f12}}\|\nabla u_\lambda(s+\tau)\|_{L^p(\R^2)}\|\nabla v_\lambda(s+\tau)\|_{L^\infty(\R^2)}\,ds\\
&\qquad+C\int_0^t\f1{(t-s)^{\f12}}\|u_\lambda(s+\tau)\|_{L^p(\R^2)}\|\Delta v_\lambda(s+\tau)\|_{L^\infty(\R^2)}\,ds\\
&\le C\,t^{-\f 12}\tau^{-(1-\f 1p)}+C\int_0^t\f1{(t-s)^{\f12}}\f1{(s+\tau)^{\f12}}\|\nabla u_\lambda(s+\tau)\|_{L^p(\R^2)}ds\\
&\qquad+C\int_0^t\f1{(t-s)^{\f12}}\f1{(s+\tau)^{1-\f1p}}\|\Delta v_\lambda(s+\tau)\|_{L^\infty(\R^2)}\,ds\,.
\end{split}
\label{eq:est nabla u lambda}
\eeq
Therefore, the function $\phi_\lambda(t,\tau):=\|\nabla u_\lambda(t+\tau)\|_{L^p(\R^2)}+\|\Delta v_\lambda(t+\tau)\|_{L^\infty(\R^2)}$ satisfies the inequality
\beq
\phi_\lambda(t,\tau)\le C\,f_p(\tau)\,t^{-\f 12}+ C\,f_p(\tau)\int_0^t\f1{(t-s)^{\f12}}\phi_\lambda(s,\tau)\,ds+C\,\int_0^t\f1{(t-s)^{\f1p+\f12}}\phi_\lambda(s,\tau)\,ds\,,
\label{eq:phi_lambda}
\eeq
for any $t>0$ and $\tau>0$, where $f_p(\tau):=(\tau^{-\f 12}+\tau^{-(1-\f 1p)})$. Applying Lemma \ref{lem:Gronwall} to \eqref{eq:phi_lambda} with respect to $t\in(0,T]$, $T>0$ arbitrarily fixed, we then get for any $\tau>0$
\beq
\|\nabla u_\lambda(t+\tau)\|_{L^p(\R^2)}+\|\Delta v_\lambda(t+\tau)\|_{L^\infty(\R^2)}\le C(\tau,T)\,t^{-\f12}\,,\qquad\ t\in(0,T]\,.
\label{est u_lambda v_lambda}
\eeq
Undoing the scaling and choosing $\tau=t=T=1$, \eqref{est u_lambda v_lambda} gives us
\[
\|\nabla u_\lambda(2)\|_{L^p(\R^2)}=\lambda^{2(\f32-\f1p)}\|\nabla u(2\lambda^2)\|_{L^p(\R^2)}\le C
\]
and 
\[
\|\Delta v_\lambda(2)\|_{L^\infty(\R^2)}=\lambda^2\|\Delta v(2\lambda^2)\|_{L^\infty(\R^2)}\le C\,,
\]
for any $\lambda>0$.
Hence, \eqref{eq:grad u} for $p>2$ and \eqref{eq:delta v}  for $r=\infty$ follow.

For $p\in[1,2]$, it is sufficient to plug the $L^\infty$ bound \eqref{est u_lambda v_lambda} for $\Delta v_\lambda$ into the r.h.s. of \eqref{eq:est nabla u lambda} to obtain, for $t\in(0,T]$ and $\tau>0$,
\[
\|\nabla u_\lambda(t+\tau)\|_{L^p(\R^2)}\le C(\tau,T)\,\tau^{-(1-\f 1p)}(t^{-\f 12}+1)+
C\,\tau^{-\f12}\int_0^t\f1{(t-s)^{\f12}}\|\nabla u_\lambda(s+\tau)\|_{L^p(\R^2)}ds\,.
\]
Applying Lemma \ref{lem:Gronwall} again and undoing the scaling as before, give us \eqref{eq:grad u}. 

Finally, taking the $L^r$ norm of \eqref{eq-delta v lambda}, with $r\in [2,\infty)$ and using \eqref{reg_effetct2_lambda}, \eqref{est u_lambda v_lambda}, 
we get
\[
\begin{split}
\|\Delta v_\lambda(t+\tau)\|_{L^r(\R^2)}&\le C\,t^{-\f 12}\tau^{-(\f12-\f 1r)}+C\int_0^t\f1{(t-s)^{\f 1p-\f1r+\f12}}\|\nabla u_\lambda(s+\tau)\|_{L^p(\R^2)}\,ds\\
&\le C\,t^{-\f 12}\tau^{-(\f12-\f 1r)}+C(\tau,T)\,t^{(\f1r-\f1p)}\,,
\end{split}
\]
where $t\in(0,T]$ and $p>2$. Hence, for any $\lambda>0$,
\[
\|\Delta v_\lambda(2)\|_{L^r(\R^2)}=\lambda^{2(1-\f1r)}\|\Delta v(2\lambda^2)\|_{L^r(\R^2)}\le C\,,
\]
and the theorem is proved.
\end{proof}
\begin{remark}\label{Rk:global existence}
Integral solutions have been studied by several authors. Global existence of such solutions in the case $\e=1$ was obtained: in~\cite{B98} with $u_0$ a finite measure with small mass  and $|\nabla v_0|\in L^2(\R^2)$; in \cite{NSU03} for  $u_0,\,v_0$ and $|\nabla v_0|$ in $(L^1\cap L^\infty)(\R^2)$, $u_0$ small in $L^1$, $|\nabla v_0|$ small in $L^1\cap L^\infty$, together with the optimal decay of $\|u(t)\|_{L^p(\R^2)}$; in  \cite{N06} with $u_0\in L^1(\R^2)$ and $|\nabla v_0|\in L^2(\R^2)$ small, together with the optimal decay rate of $\|u(t)\|_{L^p(\R^2)}$ for $p\in (4/3, 2)$; in \cite{Ferreira} with $u_0\in\dot{B}^{-2(1-\f1r)}_{r,\infty}$ such that $\sup _{ t>0 }t^{(1-1/r)}\|G(t)u_0\|_{L^r(\R^2)}$ is small  for some $r\in (1, 2)$, and $v_0$ small in the homogeneous Besov space $\dot{B} _{ \infty, \infty }^0$, together with the optimal decay rate for $\|u(t)\|_{L^p(\R^2)}$ if $p\in[r,\infty)$ and for $\|\nabla v(t)\|_{L^\infty(\R^2)}$.

In the case $\e>0$, global existence of integral solutions was proved in \cite{BiBr} for $u_0$ tempered distribution  such that $\sup _{t>0,x\in \R^2 }(t+|x|^2)|G(t)u_0(x)|$ is small  and $v_0=0$; the function $u(t)$ was then shown to be such that $\sup _{t>0,x\in\R^2}(t+|x|^2)|u(t,x)|$ is bounded. More recently, the case  $v_0=0$ was considered again in \cite{BGK}. The authors proved that for $u_0$ any finite Radon measure there exists an $\varepsilon (u_0)>0$ such that for all $\e\ge\e(u_0)$, the system has a global integral solution $(u, v)$, and $u(t)$ satisfies the optimal $L^p$ time decay rates for all $p\in[1,\infty]$.
\end{remark}

We conclude this section showing the continuous dependence of the solution $(u,v)$ given by Theorem~\ref{th:existence critical} with respect to the initial data. This continuity result  shall imply the uniqueness and the positivity of the solution itself. 

\begin{theorem}[Continuous dependence]\label{th:continuous dependence}
Let $\e>0$, $\alpha\ge0$, and let $u_0^i\in L^1(\R^2)$ and 
$v_0^i\in \dot{H}^1(\R^2)$,
$i=1,2$, be two initial data sufficiently small so that the corresponding solutions $(u^i,v^i)$ of \eqref{eq:Duhamel_u}-\eqref{eq:Duhamel_v} are global. Then, for any $p\in[1,\infty]$ and $r\in[2,\infty]$, there exists $C=C(p,r)>0$ independent of $t$, such that for $t>0$ it holds
\beq
t^{(1-\f1p)}\|u^1(t)-u^2(t)\|_{L^p(\R^2)}+t^{(\f12-\f1r)}\|\nabla v^1(t)-\nabla v^2(t)\|_{L^r(\R^2)}\le C\left(\|u_0^1-u_0^2\|_{L^1(\R^2)}+\,\|\nabla v_0^1-\nabla v_0^2\|_{L^2(\R^2)}\right)\,.
\label{in:continuity}
\eeq
\end{theorem}

\begin{corollary}[Uniqueness and positivity]\label{crl:uniqueness positivity}
The global solution $(u,v)$ given by Theorem~\ref{th:existence critical} is unique. Moreover, it is non-negative whenever $u_0$ and $v_0$ are non-negative.
\end{corollary}

\begin{proof}[Proof of Theorem \ref{th:continuous dependence}]
We shall prove the continuous dependence of the solution with respect to the initial data \eqref{in:continuity} taking advantage of the rescaled solutions $(u^i_\lambda,v^i_\lambda)$ and using the same ideas as in Proposition \ref{pr:moreestimates}. 

Let $t>0$ and $\tau>0$ be arbitrarily fixed. From \eqref{eq:Duhamel_u}, \eqref{reg_effetct1_lambda} and \eqref{reg_effetct2_lambda}, we have for any $p\ge1$
\beq\label{est:u1-u2}
\begin{split}
\|u^1_\lambda(t+\tau)-u^2_\lambda(t+\tau)\|_{L^p(\R^2)}&\le C\,(t+\tau)^{-(1-\f1p)}\,\|u_0^1-u_0^2\|_{L^1(\R^2)}\\
&\quad +C\int_0^t\f1{(t-s)^{\f12}} \|u^1_\lambda(s+\tau)-u^2_\lambda(s+\tau)\|_{L^p(\R^2)}\|\nabla v^1_\lambda(s+\tau)\|_{L^\infty(\R^2)}\,ds\\
&\quad +C\int_0^t\f1{(t-s)^{\f12}} \|u^2_\lambda(s+\tau)\|_{L^p(\R^2)}\|\nabla v^1_\lambda(s+\tau)-\nabla v^2_\lambda(s+\tau)\|_{L^\infty(\R^2)}\,ds\\
&\le C\,\tau^{-(1-\f1p)}\,\|u_0^1-u_0^2\|_{L^1(\R^2)}\\
&\quad+C\,\tau^{-\f12}\int_0^t\f1{(t-s)^{\f12}}\|u^1_\lambda(s+\tau)-u^2_\lambda(s+\tau)\|_{L^p(\R^2)}\,ds\\
&\quad +C\,\tau^{-(1-\f1p)}\int_0^t\f1{(t-s)^{\f12}}\|\nabla v^1_\lambda(s+\tau)-\nabla v^2_\lambda(s+\tau)\|_{L^\infty(\R^2)}\,ds\,.
\end{split}
\eeq
On the other hand, from \eqref{eq:Duhamel_v} and $p>2$, we obtain 
\beq
\begin{split}
\|\nabla v^1_\lambda(t+\tau)-\nabla v^2_\lambda(t+\tau)\|_{L^\infty(\R^2)}&\le C\,(t+\tau)^{-\f12}\|\nabla v_0^1-\nabla v_0^2\|_{L^2(\R^2)}\\
&\quad+C\int_0^t\f1{(t-s)^{\f1p+\f12}}\|u^1_\lambda(s+\tau)-u^2_\lambda(s+\tau)\|_{L^p(\R^2)}\, ds\,.
\label{est:grad v1-grad v2}
\end{split}
\eeq
Combining \eqref{est:u1-u2} and \eqref{est:grad v1-grad v2}, it is easy to see that the function 
\[
\phi_\lambda(t,\tau):=\|u^1_\lambda(t+\tau)-u^2_\lambda(t+\tau)\|_{L^p(\R^2)}+\|\nabla v^1_\lambda(t+\tau)-\nabla v^2_\lambda(t+\tau)\|_{L^\infty(\R^2)}\,,
\]
satisfies the inequality
\beq
\begin{split}
\phi_\lambda(t,\tau)\le C\,f_p&(\tau)(\|u_0^1-u_0^2\|_{L^1(\R^2)}+\|\nabla v_0^1-\nabla v_0^2\|_{L^2(\R^2)})\\
&+C\,f_p(\tau)\int_0^t\f1{(t-s)^{\f12}}\phi_\lambda(s,\tau)\,ds
+C\int_0^t\f1{(t-s)^{\f1p+\f12}}\phi_\lambda(s,\tau)\,ds\,,
\end{split}
\label{est:phi_continuity}
\eeq
for any $t>0$ and $\tau>0$, where $f_p(\tau):=(\tau^{-\f 12}+\tau^{-(1-\f 1p)})$. Therefore, applying the Gronwall's Lemma~\ref{lem:Gronwall} to \eqref{est:phi_continuity} with respect to  $t\in(0,T]$, as in Proposition \ref{pr:moreestimates}, we obtain 
\beq
\|u^1_\lambda(t+\tau)-u^2_\lambda(t+\tau)\|_{L^p(\R^2)}+\|\nabla v^1_\lambda(t+\tau)-\nabla v^2_\lambda(t+\tau)\|_{L^\infty(\R^2)}\le
C(T,\tau)(\|u_0^1-u_0^2\|_{L^1(\R^2)}+\|\nabla v_0^1-\nabla v_0^2\|_{L^2(\R^2)})\,.
\label{eq:est:continuity 1}
\eeq
Choosing $\tau=t=T=1$ and undoing the scaling, we get \eqref{in:continuity} for $p>2$ and $r=\infty$. 

For $p\in[1,2]$, it is sufficient to plug the $L^\infty$ bound \eqref{eq:est:continuity 1} for $\|\nabla v^1_\lambda(t+\tau)-\nabla v^2_\lambda(t+\tau)\|_{L^\infty(\R^2)}$ into the r.h.s. of \eqref{est:u1-u2}, so that for $t\in(0,T]$
\[
\begin{split}
\|u^1_\lambda(t+\tau)-u^2_\lambda(t+\tau)\|_{L^p(\R^2)}&\le
C(T,\tau)(\|u_0^1-u_0^2\|_{L^1(\R^2)}+\|\nabla v_0^1-\nabla v_0^2\|_{L^2(\R^2)})\\
&\quad+C\,\tau^{-\f12}\int_0^t\f1{(t-s)^{\f12}}\|u^1_\lambda(s+\tau)-u^2_\lambda(s+\tau)\|_{L^p(\R^2)}\,ds\,.
\end{split}
\]
Applying Lemma \ref{lem:Gronwall} again and undoing the scaling as before, give us \eqref{in:continuity} for $p\in[1,2]$ and $r~=~\infty$. 

Finally, for  any $r\in[2,\infty)$, using \eqref{eq:est:continuity 1} with $p>2$, we have for $t\in(0,T]$
\[
\begin{split}
&\|\nabla v^1_\lambda(t+\tau)-\nabla v^2_\lambda(t+\tau)\|_{L^r(\R^2)}\le 
C\,(t+\tau)^{-(\f12-\f 1r)}\|\nabla v_0^1-\nabla v_0^2\|_{L^2(\R^2)}\\
&\hskip200pt+C\int_0^t\f1{(t-s)^{\f 1p-\f1r+\f12}}\|u^1_\lambda(s+\tau)-u^2_\lambda(s+\tau)\|_{L^p(\R^2)}\,ds\\
&\qquad\le C\,\tau^{-(\f12-\f 1r)}\|\nabla v_0^1-\nabla v_0^2\|_{L^2(\R^2)}
+C(T,\tau)(\|u_0^1-u_0^2\|_{L^1(\R^2)}+\|\nabla v_0^1-\nabla v_0^2\|_{L^2(\R^2)})\,t^{(\f12+\f1r-\f1p)}\,.
\end{split}
\]
The conclusion follows as above.
\end{proof}
\begin{proof}[Proof of Corollary \ref{crl:uniqueness positivity}]
The uniqueness is an immediate consequence of the continuous dependence property \eqref{in:continuity}. Next, for $u_0\,,v_0\ge0$, it holds $v(t)\ge0$ whenever $u(t)\ge0$ and the latter follows by \eqref{in:continuity} for $p=\infty$. Indeed, let $(u_{0,n},|\nabla v_{0,n}|)\in((L^1\cap L^\infty)\times(L^2\cap L^q))(\R^2)$, $q>2$, be a sequence of non-negative smooth initial data such that  $u_{0,n}\to u_0$ in $L^1(\R^2)$ and $|\nabla v_{0,n}|\to|\nabla v_{0}| $ in $L^2(\R^2)$, as $n\to\infty$. Then, with the same technical tools used so far, it is shown that the associated global solution $(u_n,v_n)$ given by Theorem \ref{th:existence critical} satisfies, for a constant $C>0$ independent of $t$, $a>1$ arbitrarily fixed, $p\ge a$ and $r\ge q$,
\beq\label{reg_effetct+}
t^{(\f1a-\f1p)}\norm{u_n(t)}_{L^p(\R^2)}+ t^{(\f1q-\f1r)}\|\nabla v_n(t)\|_{L^r(\R^2)}\le C\,,\quad t>0\,.
\eeq
Multiplying \eqref{KSu} by $(u^-_n)^{a-1}$, where $u^-_n:=\max\{-u_n;0\}$, integrating the resulting equation over $\R^2$, and using  \eqref{reg_effetct+}, that gives a better time decay than \eqref{reg_effetct1}-\eqref{reg_effetct2} for $t\le1$, we obtain
\[
\begin{split}
\f d{dt}\|u^-_n(t)&\|^a_{L^a(\R^2)}=-4(1-a^{-1})\|\nabla(u^-_n)^{\f a2}(t)\|^2_{L^2(\R^2)}+2(a-1)\int_{\R^2}(u^-_n)^{\f a2}\nabla(u^-_n)^{\f a2}\cdot\nabla v_n\,dx\\
&\le-4(1-a^{-1})\|\nabla(u^-_n)^{\f a2}(t)\|^2_{L^2(\R^2)}
+2(a-1)\|\nabla v_n(t)\|_{L^\infty(\R^2)}\|u^-_n(t)\|^{a/2}_{L^a(\R^2)}\|\nabla(u^-_n)^{\f a2}(t)\|_{L^2(\R^2)}\\
&\le(\delta-4(1-a^{-1}))\|\nabla(u^-_n)^{\f a2}(t)\|^2_{L^2(\R^2)}+C(a,\delta)\f1{t^{2/q}}\|u^-_n(t)\|^a_{L^a(\R^2)}\,.
\end{split}
\]
Finally, choosing $0<\delta<4(1-a^{-1})$ and integrating over $(0,t)$, we get
\[
\|u^-_n(t)\|^a_{L^a(\R^2)}\le C(a,\delta)\int_0^t \f1{s^{2/q}}\|u^-_n(s)\|^a_{L^a(\R^2)}ds\,.
\]
Gronwall's lemma implies $\|u^-_n(t)\|^a_{L^a(\R^2)}=0$ for all $t>0$. Hence, $u_n(t)\ge0$ and $u(t)\ge0$ as well, thanks to \eqref{in:continuity} applied to $u_n(t)$ and $u(t)$ with  $p=\infty$, as announced. 
\end{proof}
\section{Uniqueness of self-similar solutions ($\alpha=0$)}
\label{sec:self-similar}
The invariance of system \eqref{KSu}-\eqref{KSv} with $\alpha=0$ under the action of the space-time scaling \eqref{eq:scaling}, naturally raises the question of the existence of solutions that are, themselves, invariant under the same scaling, i.e. the existence of the uniparametric family $(u_M,v_M)$ with
\beq\label{eq:SS_UV}
u_M(x,t)=\f1tU_M\left(\f x{\sqrt t}\right)\qquad\text{and}\qquad v_M(x,t)=V_M\left(\f x{\sqrt t}\right)\,,\quad t>0\,,\ x\in\R^2\,,
\eeq
indexed by the conserved mass $M$ of $u_M$.  

The analysis of this class of solutions has been carried on, following different techniques and approaches, in \cite{B95,BKLN} for the parabolic-elliptic system, and in \cite{B98,B06,BCD,Ferreira, MMY99, N06, NSY02,Y01} for the parabolic-parabolic case. Recently, in \cite{BCD} the authors refined the existing results concerning positive integrable self-similar solutions, and pointed out the difference between the parabolic-elliptic case, where \eqref{eq:SS_UV} exists iff $M<8\pi$ and are unique \cite{BKLN},  and the parabolic-parabolic case. Indeed, they proved that (see \cite{BCD} Theorem~4): for any $\e>0$, there exists a finite threshold $M^*(\e)\ge8\pi$, such that  system \eqref{KSu}-\eqref{KSv} with $\alpha=0$ has no positive self-similar solutions \eqref{eq:SS_UV} with profile $(U_M,V_M)\in (C^2_0(\R^2))^2$ if $M>M^*(\e)$ and has at least one positive solution with profile $(U_M,V_M)\in (C^2_0(\R^2))^2$ if $M\in(0,M^*(\e))$ and $M^*(\e)=8\pi$, or if $M\in(0,M^*(\e)]$ and $M^*(\e)>8\pi$.
  Moreover, there exist $\e^*$ and $\e^{**}$ with $\f12\le\e^*\le\e^{**}$ such that  : $M^*(\e)=8\pi$ if $\e\in(0,\e^*]$ and $M^*(\e)>8\pi$ if $\e>\e^{**}$. Finally, when the threshold $M^*(\e)>8\pi$, there are at least two positive self-similar solutions \eqref{eq:SS_UV} for any $M\in(8\pi,M^*(\e))$. The identity $\e^*=\e^{**}$ is not proved but conjectured and would put the described behavior in a dichotomy. On the other hand, when $M^*(\e)=8\pi$, it is still an open problem if there is or not a positive integrable self-similar solution with $M=M^*(\e)$. 

Whatever is the landscape of the family \eqref{eq:SS_UV}, the uniqueness of $(u_M,v_M)$ for $\e>0$  arbitrary and~$M$ below a threshold (that has to depend on $\e$)  is still an open problem. This section is devoted to the proof of  the following uniqueness result.
\begin{theorem}[Uniqueness]\label{th:uniqueness_self-similar}
For any fixed $\e>0$, there exists $\widetilde M(\e)\in[4\pi,8\pi]$, defined in \eqref{eq:tilde M(e)}, such that  for any $M<\widetilde M(\e)$ the Keller-Segel system \eqref{KSu}-\eqref{KSv} has a unique positive self-similar solution with profile $(U_M,V_M)\in (C^2_0(\R^2))^2$. Furthermore, if $\e\le\f12$, $\widetilde M(\e)=8\pi$.
\end{theorem}

It follows by Theorem \ref{th:uniqueness_self-similar} and \cite{BCD}, that the case $\e\in(0,\f12]$ is completely understood: there exists a unique positive and smooth self-similar solution iff the associated mass $M$ is below~$8\pi$. Moreover, this uniparametric family of solutions describes the long time behavior of the global integral solutions (see Theorem \ref{th:longtime-alpha=0}). In other words, the parabolic-parabolic Keller-Segel system behaves like the parabolic-elliptic one when $\e\le\f12$. Theorem \ref{th:uniqueness_self-similar} and the results in \cite{BCD} are illustrated in Figure \ref{fig:Mtilde}.

In addition to the uniqueness result above, we shall prove the continuity of $(u_M,v_M)$ with respect to $M$, a property fundamental in our investigation of the long time behavior of global solution.
\begin{proposition}[Continuity with respect to  $M$]\label{pr:continuity}
Let $\e>0$ and $\widetilde M(\e)$ be given by Theorem \ref{th:uniqueness_self-similar}. Let $M\in(0,\widetilde M(\e))$ and $M_n\in(0,\widetilde M(\e))$ be a sequence such that  $M_n\to M$ as $n\to\infty$. Finally, let $(U_{M_n}, V_{M_n})$ and $(U_M,V_M)$ be the profiles of the unique self-similar solutions corresponding to $M_n$ and $M$ respectively. Then, for any $p\in[1,\infty)$ and $r\in[2,\infty)$,
\beq
(U_{M_n},|\nabla V_{M_n}|)\to (U_M,|\nabla V_M|)\quad\text{in}\quad L^p(\R^2)\times L^r(\R^2)\quad\text{as }n\to\infty\,.
\label{eq:convergence}
\eeq
\end{proposition}

To begin with, let us recall that $(u_M,v_M)$ is a self-similar solution of \eqref{KSu}-\eqref{KSv} iff its profile $(U_M,V_M)$ satisfies the elliptic system 
%
\begin{eqnarray}
\label{KSuSS}
&&\Delta U+\nabla\cdot\left(U\,\nabla\left(\f{|\xi|^2}4-V\right)\right)=0\,,\\
\label{KSvSS}
&&\Delta V+\frac\e{2}\,\xi\cdot\nabla V+U=0\,,
\end{eqnarray}
where $\xi=x/\sqrt t$ and the differential operators are taken with respect to $\xi$. Concerning \eqref{KSuSS}-\eqref{KSvSS},  it has been proved in \cite{NSY02} that any  solution $(U,V)$ in the space $(C^2_0(\R^2))^2$, (i.e. decaying to zero at infinity), are necessarily positive, radially symmetric about the origin, decreasing and satisfies $U(\xi)~=~\sigma\, {\rm e}^{V(\xi)}{\rm e}^{-|\xi|^2/4}$, for some positive constant~$\sigma$. Moreover,
\[
V(\xi)\le C\,{\rm e}^{-\min\{1,\e\}|\xi|^2/4}\,,
\label{KSvSS2}
\]
where $C$ is any positive constant such that  $C\,\min\{1,\e\}\ge\sigma\,{\rm e}^{\|V\|_\infty}$, \cite{NSY02}. Consequently,  $U$ and $V$ are integrable and we are allowed to consider the associated cumulated densities defined by
\begin{eqnarray}
\label{phi}
&&\phi(y):=\frac 1{2\,\pi}\int_{B(0,\sqrt y)}U(\xi)d\xi=\int_0^{\sqrt y}r\,U(r)dr\,,\\
\label{psi}
&&\psi(y):=\frac 1{2\,\pi}\int_{B(0,\sqrt y)}V(\xi)d\xi=\int_0^{\sqrt y}r\,V(r)dr\,,
\end{eqnarray}
where $r=|\xi|$. Furthermore, using the radial formulation of \eqref{KSuSS}-\eqref{KSvSS} and definitions \eqref{phi}-\eqref{psi}, it is easy to see that the cumulated densities $(\phi,\psi)$ satisfies the ODE system
\begin{eqnarray*}
&& \phi''+\frac{1}{4}\phi'-2\phi'\psi''=0\,,\\
&& 4y\psi''+\e y\psi'-\e\psi+\phi=0\,,
\end{eqnarray*}
which reads, defining $S(y):=4\,(\psi(y)-y\,\psi'(y))'=-4\,y\,\psi''(y)$, as following
\begin{eqnarray}
\label{phi2}
&&\phi''+\frac14\,\phi'+\frac1{2\,y}\,\phi'S=0\,,\\
\label{S}
&&S'+\frac\e 4\,S=\phi'\,.
\end{eqnarray}

System  \eqref{phi2}-\eqref{S}, endowed with the natural initial conditions
\beq
\phi(0)=0, \quad \phi'(0)=a>0 \; \mbox{ and } \; S(0)=0,
\label{ic-a}
\eeq
becomes a shooting parameter problem, with the shooting parameter $a>0$ directly related to the concentration of $U$ around the origin by the identity $a=\phi'(0)=\frac{U(0)}{2}$. It has been analyzed in \cite{BCD}, where the authors proved that for any $(a,\e)\in\R^2_+$ there exists a unique positive solution $(\phi,S)\in C^2[0,\infty)\times C^1[0,\infty)$ of \eqref{phi2}-\eqref{S}-\eqref{ic-a}. They also proved that the map $a\mapsto(\phi,S)$ is continuous, $\phi$ is a strictly increasing and concave function on $(0,\infty)$ and the following estimates (among others) hold true for~$y>0$
\beq
0<S(y)\leq \frac{\min\{1,\e\}\,a\,y}{(\min\{1,\e\}+a)\,{\rm e}^{\min\{1,\e\}\frac{y}{4}}-a}\,,
\label{upperS}
\eeq
\beq
0<\phi'(y)\le a\,{\rm e}^{-y/4}\,,
\label{est:phi'infty}
\eeq
\beq
\f {M(a,\e)}{2\pi}\left(1-{\rm e}^{-y/4}\right)\le\phi(y)\le\f {M(a,\e)}{2\pi}\,,
\label{est:phi}
\eeq
where
\[
\f{M(a,\e)}{2\pi}:=\phi(\infty)=\lim_{y\rightarrow \infty}\phi(y)\,.
\label{bc}
\]
In addition, with the threshold $M^*(\e):=\sup_{a>0}M(a,\e)$ introduced at the beginning of this section, the map $a\mapsto M(a,\e)$ is continuous from $\R_+$ to $(0,M^*(\e))$ if $M^*(\e)=8\pi$ and from $\R_+$ to $(0,M^*(\e)]$ if $M^*(\e)>8\pi$. That threshold is proved to be finite since $M(a,\e)$ is upper bounded by a constant independent on $a$, for all $\e>0$. We also have, for all fixed $\e>0$ and $a>0$, that \cite{BCD}
\beq\label{eq:est a w.r.t. M}
\f{M(a,\e)}{8\pi}\ge\f{a\,\min\{1,\e\}}{a+\min\{1,\e\}}
\eeq
and
\beq
\lim_{a\to\infty}M(a,\e)=8\pi\,.
\label{eq:Matinfinity}
\eeq 
Finally, the proposition below will be fundamental in the sequel. 
\begin{proposition}[\cite{BCD}]
Let $\e>0$ and define
\beq\label{eq:A(e)}
A(\e):=\left\{
\begin{split}
+\infty\qquad\qquad\qquad&\text{if }\e\le\f12\\
\min\{\e,1\}\,\f{{\rm e}^{1-\f1{2\,\e}}}{2\,\e-{\rm e}^{1-\f1{2\,\e}}}\quad\quad&\text{if }\e>\f12
\end{split}
\right.
\eeq
If $a<\max\{A(\e),1\}$, then $\e\,S(y)<2$ for all $y>0$ and $M(a,\e)<8\pi\,\min\{1,a\}$. 
\label{pr:est_S<2}
\end{proposition}

Coming back to the self-similar solutions, to each solution $(u_M,v_M)$ with profile $(U_M,V_M)\in(C^2_0(\R^2))^2$, it corresponds a solution $(\phi,S)\in C^2[0,\infty)\times C^1[0,\infty)$ of \eqref{phi2}-\eqref{S}-\eqref{ic-a} with $a=U_M(0)/2$ and $M=M(a,\e)$ and conversely. Therefore, the uniqueness issue of the  self-similar solution corresponding to given $M>0$ and $\e>0$ translates into the uniqueness issue of the solution of the boundary value problem obtained associating to the ODE system  \eqref{phi2}-\eqref{S} the boundary conditions 
\[
\phi(0)=0\,,\quad \phi(\infty)=M\quad\text{and}\quad S(0)=0\,.
\]
As a consequence of the results obtained in \cite{BCD} and recalled so far, it is clear that $M<8\pi$ is a necessary condition for the uniqueness, whatever the value of  $\e>0$ is. We are able to prove that $M<8\pi$ is also a sufficient condition in the case $\e\le\f12$. This is a direct consequence of Proposition \ref{pr:est_S<2} (used in the lemma below) implying that $M(a,\e)<8\pi$ for any positive $a$, if $\e\le\f12$. On the other hand, when $\e>\f12$, the condition $M(a,\e)<8\pi$ is guaranteed  imposing a finite upper bound (depending on $\e$) on the shooting parameter~$a$.
Unfortunately, due to the poor informations that we have on the map $a\mapsto M(a,\e)$, we are not able to prove that this upper bound on $a$ is  optimal.

We shall proceed hereafter in proving that two solutions of the shooting problem \eqref{phi2}-\eqref{S}-\eqref{ic-a} do not cross under hypothesis of Proposition \ref{pr:est_S<2}. Theorem \ref{th:uniqueness_self-similar} will be an immediate  consequence. 

\begin{lemma}\label{lm:uniquennes_self-similar}
Let $\e>0$ and let $(\phi_1,S_1), (\phi_2,S_2)\in C^2[0,\infty)\times C^1[0,\infty)$ be two solutions  of \eqref{phi2}-\eqref{S}-\eqref{ic-a} corresponding to the shooting parameters $a_1$ and $a_2$ respectively. Assume $a_1\ne a_2$ and $a_i<\max\{A(\e),1\}$, $i=1,2$. Then, $\phi_1$ and  $\phi_2$ do not intersect in $(0,\infty]$. In particular $\phi_1(\infty)\ne\phi_2(\infty)$. 
\end{lemma}
\begin{proof} We may assume without loss of generality that $a_1>a_2$.\\
{\sl First step :} we shall prove that $\phi_{1}(y)>\phi_{2}(y)$ for all $y>0$. Indeed, following \cite{BKLN}, let 
\[
y_{0}:=\sup\{y>0\ \text{ such that  }\ \phi_{1}(z)>\phi_{2}(z)\ \text{ for all } 0<z< y\}\,.
\]
By the assumption above on $a_i=\phi_{i}'(0)$ and the regularity of each $\phi_i$, it holds that $y_0>0$. Assume by contradiction that $y_{0}<\infty$. Then, 
\beq
\phi_{1}(y)>\phi_{2}(y)\quad\text{for all } 0<y<y_{0}\,,\quad\phi_{1}(y_{0})=\phi_{2}(y_{0})\quad\text{and} \quad\phi_{1}'(y_{0})\leq\phi_{2}'(y_{0})\,.
\label{eq:contradiction}
\eeq

Next, let us observe that, owing to the identity (see equation \eqref{S})
\beq\label{S-phi'}
S(y)={\rm e}^{-\e\,y/4}\int_{0}^y{\rm e}^{\e\, z/4}\phi'(z)\,dz\,,
\eeq
system \eqref{phi2}-\eqref{S} can also be equivalently written as a single nonlocal integro-differential equation for~$\phi'$, namely
\begin{equation}
\phi''+\frac14\,\phi'+\frac1{2\,y}\,\phi'\,{\rm e}^{-\e\,y/4}\int_0^y{\rm e}^{\e\,z/4}\,\phi'(z)dz=0\,.
\label{phi3}
\end{equation}
Multiplying equation \eqref{phi3} by $y$, integrating the resulting equation over $[0,y_0]$ and using the initial condition $\phi_i(0)=0$, we obtain that each solution $\phi_i$ satisfies 
\beq
y_0\,\phi_i'(y_0)-\phi_i(y_0)+\f14y_0\,\phi_i(y_0)-\f14\int_0^{y_0}\phi_i(y)\,dy+\f12J_i(y_0)=0\,,
\label{eq-phi''int}
\eeq
where
\[
J_i(y):=\int_0^{y}\phi'_i(z)\,{\rm e}^{-\e\,z/4}\int_0^z{\rm e}^{\e\,\xi/4}\,\phi'_i(\xi)d\xi\,dy\,.
\label{def:J}
\]
Moreover, using \eqref{eq:contradiction}, the difference $(\phi_1-\phi_2)$ satisfies
\beq
y_0(\phi_1'-\phi_2')(y_0)-\f14\int_0^{y_0}(\phi_1-\phi_2)(y)\,dy+\f12(J_1-J_2)(y_0)=0\,.
\label{eq:phi1-phi2}
\eeq
It is worth noticing that $(J_1-J_2)(y_0)=0$ if $\e=0$. In that case, the contradiction follows directly from the sign of the remaining two terms in \eqref{eq:phi1-phi2}. Since here $\e>0$, we have to argue deeply in order to control from above the nonzero term $(J_1-J_2)(y_0)$.

Let $f_i(y):=\int_0^y{\rm e}^{\e\,z/4}\,\phi_i'(z)dz={\rm e}^{\e\,y/4}S_i(y)$, (see \eqref{S-phi'}), so that $J_i(y)$ reads as
\beq
J_i(y)=\f12\int_0^{y}{\rm e}^{-\e\,z/2}(f_i^2)'(z)\,dz=\f12\,{\rm e}^{-\e\,y/2}f_i^2(y)
+\f\e4\int_0^{y}{\rm e}^{-\e\,z/2}f_i^2(z)\,dz\,.
\label{def2:J}
\eeq
For all $y>0$, it holds
\[
(f_1-f_2)(y)=\int_0^y{\rm e}^{\e\,z/4}\,(\phi_1-\phi_2)'(z)dz={\rm e}^{\e\,y/4}\,(\phi_1-\phi_2)(y)
-\f\e4\int_0^y{\rm e}^{\e\,z/4}\,(\phi_1-\phi_2)(z)dz\,,
\]
and, owing to \eqref{eq:contradiction}, 
\[
(f_1-f_2)(y_0)=-\f\e4\int_0^{y_0}{\rm e}^{\e\,z/4}\,(\phi_1-\phi_2)(z)dz<0\,.
\] 
Consequently, the difference $(J_1-J_2)(y_0)$ writes as
\beq
\begin{split}
(J_1-J_2)(y_0)&=\f12\,{\rm e}^{-\e\,y_0/2}(f_1+f_2)(y_0)(f_1-f_2)(y_0)
+\f\e4\int_0^{y_0}{\rm e}^{-\e\,y/2}(f_1+f_2)(y)(f_1-f_2)(y)\,dy\\
&=-\f\e8\,{\rm e}^{-\e\,y_0/2}(f_1+f_2)(y_0)\int_0^{y_0}{\rm e}^{\e\,y/4}\,(\phi_1-\phi_2)(y)\,dy\\
&\quad+\f\e4\int_0^{y_0}{\rm e}^{-\e\,y/4}\,(f_1+f_2)(y)\,(\phi_1-\phi_2)(y)\,dy\\
&\quad-\f{\e^2}{16}\int_0^{y_0}{\rm e}^{-\e\,y/2}(f_1+f_2)(y)\int_0^y{\rm e}^{\e\,z/4}\,(\phi_1-\phi_2)(z)\,dz\,dy\,.
\end{split}
\label{eq:J1-J2}
\eeq
Finally, using the increasing behavior of each $f_i$, the double integral term in the r.h.s. of \eqref{eq:J1-J2} can be estimated as follows 
\beq
\begin{split}
\int_0^{y_0}&{\rm e}^{-\e\,y/2}(f_1+f_2)(y)\int_0^y{\rm e}^{\e\,z/4}\,(\phi_1-\phi_2)(z)\,dz\,dy=
\int_0^{y_0}{\rm e}^{\e\,z/4}\,(\phi_1-\phi_2)(z)\int_z^{y_0}{\rm e}^{-\e\,y/2}(f_1+f_2)(y)\,dy\,dz\\
&\ge\int_0^{y_0}{\rm e}^{\e\,z/4}\,(\phi_1-\phi_2)(z)(f_1+f_2)(z)\int_z^{y_0}{\rm e}^{-\e\,y/2}\,dy\\
&=\f2\e\int_0^{y_0}{\rm e}^{-\e\,z/4}\,(\phi_1-\phi_2)(z)(f_1+f_2)(z)\,dz
-\f2\e\,{\rm e}^{-\e\,y_0/2}\int_0^{y_0}{\rm e}^{\e\,z/4}\,(\phi_1-\phi_2)(z)(f_1+f_2)(z)\,dz\\
&\ge\f2\e\int_0^{y_0}{\rm e}^{-\e\,z/4}\,(\phi_1-\phi_2)(z)(f_1+f_2)(z)\,dz
-\f2\e\,{\rm e}^{-\e\,y_0/2}(f_1+f_2)(y_0)\int_0^{y_0}{\rm e}^{\e\,z/4}\,(\phi_1-\phi_2)(z)\,dz\,.
\label{est:J1-J2}
\end{split}
\eeq
Plugging \eqref{est:J1-J2} into \eqref{eq:J1-J2} and rearranging the terms, we obtain the estimate
\beq
(J_1-J_2)(y_0)\le\f\e8\int_0^{y_0}{\rm e}^{-\e\,y/4}\,(f_1+f_2)(y)\,(\phi_1-\phi_2)(y)\,dy\,,
\label{est2:J1-J2}
\eeq
that in turn, plugged into identity \eqref{eq:phi1-phi2}, gives us 
\beq
y_0(\phi_1'-\phi_2')(y_0)-\f14\int_0^{y_0}(\phi_1-\phi_2)(y)\,dy+\f\e{16}\int_0^{y_0}{\rm e}^{-\e\,y/4}\,(f_1+f_2)(y)\,(\phi_1-\phi_2)(y)\,dy\ge0\,.
\label{eq:fin-contradiction}
\eeq
Then, since by \eqref{S-phi'} and Proposition \ref{pr:est_S<2} it holds 
\[
\f\e{16}\,{\rm e}^{-\e\,y/4}\,(f_1+f_2)(y)=\f\e{16}\,(S_1+S_2)(y)<\f14\,,
\] 
inequality \eqref{eq:fin-contradiction} implies the contradiction, taking into account that $y_0(\phi_1'-\phi_2')(y_0)\le0$.
\medskip

{\sl Second step :} we shall prove that $\phi_1$ and $\phi_2$ do not cross at infinity, i.e. $\phi_{1}(\infty)>\phi_{2}(\infty)$. From equation \eqref{eq-phi''int}, true for any $y>0$, we have
\beq
y(\phi_{1}'-\phi_{2}')(y)-(\phi_{1}-\phi_{2})(y)+\f y4(\phi_{1}-\phi_{2})(y)
-\f14\int_0^y(\phi_{1}-\phi_{2})(z)\,dz+\f12(J_1-J_2)(y)=0\,.
\label{eq:phi1-phi2-infty}
\eeq
Next, let $\phi_{1}(\infty)=\frac{M_{1}}{2\pi}$ and $\phi_{2}(\infty)=\frac{M_{2}}{2\pi}$. From the previous step we know that $M_{1}\geq M_{2}$.  Assume  $M_1=M_2=M$. By \eqref{est:phi}, it follows that, for all $y>0$,
$$
0<\phi_{1}(y)-\phi_{2}(y)\leq \f M{2\pi}\,{\rm e}^{-\frac{y}{4}}.
$$
Hence, $(\phi_{1}-\phi_{2})\in L^1(0,\infty)$ and $\lim_{y\rightarrow \infty}(\phi_{1}-\phi_{2})(y)=\lim_{y\rightarrow \infty}y\,(\phi_{1}-\phi_{2})(y)=0$. Furthermore, by~\eqref{est:phi'infty} it follows that $\lim_{y\rightarrow \infty}y\,(\phi'_{1}-\phi'_{2})(y)=0$, while by \eqref{def2:J}, the identity ${\rm e}^{-\e\,y/2}f_i^2(y)=S_i^2(y)$ and estimate \eqref{upperS}, we get
\[
\lim_{y\rightarrow \infty}J_i(y)=\f\e4\int_0^{\infty}{\rm e}^{-\e\,y/2}f_i^2(y)\,dy<\infty\,.
\]
Therefore, we are allowed to let $y\to\infty$ in \eqref{eq:phi1-phi2-infty} to obtain
\beq
-\f14\int_0^\infty(\phi_{1}-\phi_{2})(z)\,dz+\f12(J_1-J_2)(\infty)=0\,.
\label{eq:phi1-phi2-infty-bis}
\eeq
Proceeding exactly as in \eqref{eq:J1-J2} and \eqref{est:J1-J2}, the difference $(J_1-J_2)(\infty)$ can be estimated as follows
\[
(J_1-J_2)(\infty)\le\f\e8\int_0^{\infty}{\rm e}^{-\e\,y/4}\,(f_1+f_2)(y)\,(\phi_1-\phi_2)(y)\,dy\,,
\]
the equivalent of \eqref{est2:J1-J2} for $y_0\to\infty$. Finally, plugging the latter estimate into \eqref{eq:phi1-phi2-infty-bis}, we obtain
\[
-\f14\int_0^{\infty}(\phi_1-\phi_2)(y)\,dy+\f\e{16}\int_0^{\infty}{\rm e}^{-\e\,y/4}\,(f_1+f_2)(y)\,(\phi_1-\phi_2)(y)\,dy\ge0\,,
\]
and the contradiction follows as in the first step.
\end{proof}
We are now able to prove Theorem \ref{th:uniqueness_self-similar} and Proposition \ref{pr:continuity}.

\begin{proof}[Proof of Theorem \ref{th:uniqueness_self-similar}]
For $\e>0$, let $m(\e):=\sup\{M(a,\e)\,;\,0<a<\max\{A(\e),1\}\}$ and
\beq\label{eq:tilde M(e)}
\widetilde M(\e):=\left\{
\begin{split}
&8\pi\qquad\text{if }\e\in(0,\f12]\\
&4\pi\,{\rm e}^{1-\f1{2\e}}\qquad\text{if }\e\in(\f12,1)\\
&4\pi\max\{1,\e^{-1}\,{\rm e}^{1-\f1{2\e}}\}\qquad\text{if }\e\ge1
\end{split}
\right.
\eeq
Lemma \ref{lm:uniquennes_self-similar} shows that the continuous map $a\in(0,\max\{A(\e),1\})\mapsto M(a,\e)\in(0,m(\e))$ is strictly increasing. Therefore, for any $M<\min\{m(\e),\widetilde M(\e)\}$ there exists a unique positive self-similar solution with profile $(U_M,V_M)\in(C^2_0(\R^2))^2$ and corresponding shooting parameter satisfying $a<\max\{A(\e),1\}$. 

Next, if $\e\le\f12$, owing to \eqref{eq:Matinfinity}-\eqref{eq:A(e)}, it holds $A(\e)=+\infty$ and $m(\e)=\widetilde M(\e)=8\pi$ and the theorem follows in that case. On the other hand, if $\e\in(\f12,1)$, by $M<\widetilde M(\e)$ and \eqref{eq:est a w.r.t. M} we have that the corresponding shooting parameter satisfies 
\[
a\le\f{\e\,M}{8\pi\e-M}<A(\e)\,.
\]
Therefore, $\min\{m(\e),\widetilde M(\e)\}=\widetilde M(\e)$ and the theorem is proved also in that case. Finally, if $\e\ge1$, the proof follows exactly as in the previous case. 
\end{proof}

\begin{proof}[Proof of Proposition \ref{pr:continuity}]
Let $(\phi_n,S_n)$ be the solution of \eqref{phi2}-\eqref{S}-\eqref{ic-a} with $a_n=U_{M_n}(0)/2$ and $M_n=M(a_n,\e)$ corresponding to $(U_{M_n},V_{M_n})$. Similarly, let $(\phi,S)$ be the solution of \eqref{phi2}-\eqref{S}-\eqref{ic-a} with $a=U_M(0)/2$ and $M=M(a,\e)$ corresponding to $(U_M,V_M)$. It has been proved above that $a_n,\,a \in(0,\max\{A(\e),1\})$. Then, by the strictly increasing behaviour of the continuous map $a\in(0,\max\{A(\e),1\})\mapsto M(a,\e)\in(0,m(\e))$, the inverse map is continuous and $\lim_{n\to\infty}a_n=a$.

In order to obtain the desired continuity result \eqref{eq:convergence} for $p=1$ and $r=2$, we shall prove that
\[
(\phi'_n\,,S_n)\to(\phi',S)\qquad\text{in }(L^1(0,\infty))^2\quad\text{as}\quad n\to\infty\,,
\]
since, by the definitions of the cumulated densities, the radial symmetry of the profiles and estimate \eqref{upperS}, it follows easily that
\[
\|U_{M_n}-U_M\|_{L^1(\R^2)}=2\pi\,\|\phi'_n-\phi'\|_{L^1(0,\infty)}
\]
and
\[
\|\nabla V_{M_n}-\nabla V_M\|^2_{L^2(\R^2)}=2\pi\,\|(S_n-S)\,y^{-1/2}\|^2_{L^2(0,\infty)}\le
2\pi(a_n+a)\|S_n-S\|_{L^1(0,\infty)}\,.
\]

From the inequality (see \cite{BCD} Theorem 2)
\[
|\log\phi'_n(y)-\log\phi'(y)|\le{\rm e}^{C(n,\e)}|\log a_n-\log a|\,,\qquad y>0\,,
\]
where $C(n,\e)=2\,\f{\log\e}{\e-1}\,{\rm e}^{\max\{\log a_n,\,\log a\}}$, it follows that $\phi'_n\to\phi'$ as $n\to\infty$ uniformly on $(0,\infty)$. Owing to estimate \eqref{est:phi'infty}, we are allowed to apply the Lebesgue's dominated convergence Theorem to obtain the converge of $\phi'_n$ toward $\phi'$ in $L^1(0,\infty)$. The converge of $S_n$ toward $S$ in $L^1(0,\infty)$ follows in the same way, using identity \eqref{S-phi'} and estimate \eqref{upperS}.

Next, recalling that $U_{M_n}$ and $U_M$ are positive radially symmetric about the origin and decreasing functions, we have, for any $\e>0$ and $n$ sufficiently large,
\beq\label{eq:UMn bound}
\|U_{M_n}\|_{L^\infty(\R^2)}=U_{M_n}(0)=2\,a_n\le2(a+1)\quad\text{and}\quad
\|U_M\|_{L^\infty(\R^2)}=U_M(0)=2\,a\,.
\eeq
On the other hand, since $S(y)=-4\,y\,\psi''(y)$, by the definition \eqref{psi} of $\psi$ and estimate \eqref{upperS} again, it holds, for any $\e>0$ and $n$ sufficiently large,
\beq\label{eq:gradVMn bound}
\|\nabla V_{M_n}\|_{L^\infty(\R^2)}=\sup_{y>0}\frac{S_n(y)}{\sqrt{y}}\le 4+a_n\le 4+(a+1)\quad\hbox{and}\quad
\|\nabla V_{M}\|_{L^\infty(\R^2)}\le4+a\,.
\eeq
Then, \eqref{eq:convergence} for $p\in(1,\infty)$ and $r\in(2,\infty)$ follows by interpolation and the proved convergence for $p=1$ and $r=2$.
\end{proof}
\begin{remark}
As a byproduct of the previous results, we obtain that the map $a\mapsto M(a,\e)$ is strictly increasing from $\R_+$ to $[0,8\pi)$, if $\e\le\f12$. 
\end{remark}
\begin{figure}[htp]
\begin{center}
\includegraphics[width=10cm,height=5cm]{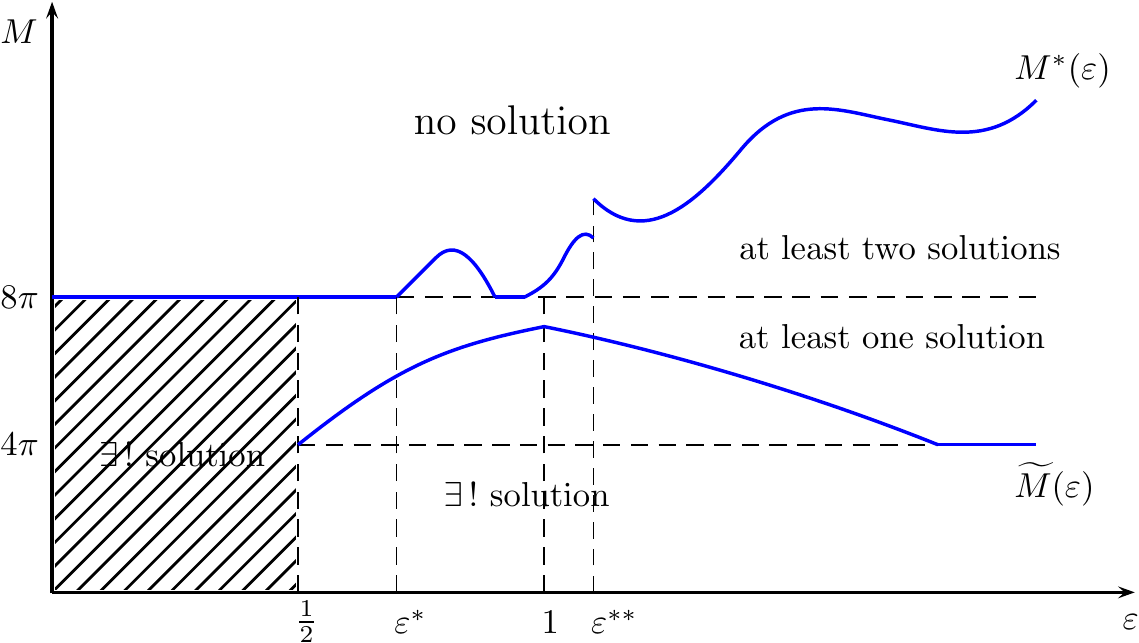}
\end{center}
\caption{range of existence and uniqueness of self-similar solutions $(u_M,v_M)$.
}
\label{fig:Mtilde}
\end{figure}
\begin{remark}\label{rk:comparison}
The profiles $U_M$ and $V_M$ are in all the $L^p(\R^2)$ spaces for all $p\in [1, \infty]$, as a consequence of their exponential decay. The corresponding self-similar solutions satisfy as $t\to 0^+$ : $u_M(t)\rightharpoonup  M\delta _0$ in the sense of measures, and $\|v_M(t)\|_{L^p(\R^2)}=t^{1/p}\|V_M\|_{L^p(\R^2)}\to0$ for all $p\in[1,\infty)$. 
Moreover,
\[
\|\nabla v_M(t)\|_{L^p(\R^2)}=t^{\f1p-\f12}\|\nabla V_M\|_{L^p(\R^2)}=t^{\f1p-\f12}\pi^{1/p}\|y^{-1/2}S(y)\|_{L^p(0,\infty)} \to 0\,,\quad \text{as } t\to 0^+\,,
\]
for all $p\in [1, 2)$. Therefore, the initial data of the self-similar solutions constructed in \cite{BCD} and considered here is compatible with the initial data of the self-similar solutions whose existence and uniqueness has been obtained in \cite{Ferreira,N06} for $\e=1$ under a smallness condition.
\end{remark}
\section{Long time behavior : the case $\alpha=0$}
\label{sec:alpha0}
In order to prove that in the case $\alpha=0$, non-negative global integral solutions behave like self-similar solutions for large $t$, we introduce the following space-time rescaled functions $(\tilde u,\tilde v)$ 
%
\beq
u(x,t)=\f1{(t+1)}\tilde u\left(\f x{\sqrt{t+1}},\log(t+1)\right)\quad\hbox{and}\quad v(x,t)=\tilde v\left(\f x{\sqrt{t+1}},\log(t+1)\right)\,,
\label{eq:tilde_uv}  
\eeq
or equivalently
\beq
\tilde u(\xi,s)=e^s\,u(\xi\,e^{\f s2},e^s-1)\quad\hbox{and}\quad \tilde v(\xi,s)=v(\xi\,e^{\f s2},e^s-1)\,,
\label{eq:tilde_uv 2}  
\eeq
where $\xi=x/\sqrt{t+1}$ and $s=\log(t+1)$. Then, $(\tilde u,\tilde v)$ satisfies the parabolic-parabolic system
\begin{eqnarray}
\label{KS_tilde_u}
\tilde u_s&=&\Delta\tilde u+\f\xi2\cdot\nabla\tilde u+\tilde u-\nabla\cdot\left(\tilde u\nabla\tilde v\right)\,,\\
\label{KS_tilde_v}
\e\,\tilde v_s&=&\Delta \tilde v+\e\,\f\xi2\cdot\nabla \tilde v+\tilde u\,,
\end{eqnarray}
where the differential operators are taken with respect to $\xi$. Moreover, $\tilde u(\xi,0)=u_0(\xi)$, $\tilde v(\xi,0)=v_0(\xi)$ and the mass of $\tilde u$ is conserved and is equal to the mass of $u$. 

The interest of the above change of variables  is that the stationary solutions $(U,V)$ of system \eqref{KS_tilde_u}-\eqref{KS_tilde_v} are solutions of the elliptic system \eqref{KSuSS}-\eqref{KSvSS}. Therefore, if one proves on the one hand that the solution $(\tilde u,\tilde v)$ of \eqref{KS_tilde_u}-\eqref{KS_tilde_v} converge toward a stationary solution $(U,V)$, i.e. satisfies
\[
\lim_{s\rightarrow \infty}\|\tilde u(s)- U\|_{L^p(\R^2)}\quad\text{and}\quad
\lim_{s\rightarrow \infty}\|\nabla\tilde v(s)-\nabla V\|_{L^r(\R^2)}\,,
\]
for some $p$ and $r$, and on the other hand that $(U,V)$ is the unique solution $(U_M,V_M)$ of \eqref{KSuSS}-\eqref{KSvSS} corresponding to $M=\int_{\R^2}u_0$, then undoing the change of variable, it follows  
\beq
\lim_{t\rightarrow \infty}t^{1-\frac{1}{p}}\|u(t)-u_M(t)\|_{L^p(\R^2)}=0\quad\text{and}\quad
\lim_{t\rightarrow \infty}t^{\frac{1}{2}-\frac{1}{r}}\|\nabla v(t)-\nabla v_M(t)\|_{L^r(\R^2)}=0\,,
\label{lim_uMvM}
\eeq
where $(u_M,v_M)$ is the self-similar solution \eqref{eq:SS_UV} with profile $(U_M,V_M)$. Equally important is also the fact that this change of variables allows to work with differential operators having strong compactness properties.

This technique is nowadays classical and has been exploited for instance in \cite{EK87b,O87} and \cite{EZ91a} for the analysis of the long-time behaviors of global solutions of the non-linear heat equation and of a convection-diffusion equation respectively. However, its application to the Keller-Segel system \eqref{KSu}-\eqref{KSv} is not straightforward and it is new, to the best of our knowledge. Before stating our main results, we shall introduce the functional framework naturally associated to system \eqref{KS_tilde_u}-\eqref{KS_tilde_v}. 

Let us consider the following weighted spaces 
\[
L^2(K_\theta):=\{f: \|f\|_{L^2(K_\theta)}^2=\int_{\R^2}|f(\xi)|^2K_\theta(\xi)\,d\xi<\infty\},
\]
and
\[
H^k(K_\theta):=\{f\in L^2(K_\theta): \|f\|_{H^k(K_\theta)}^2=\sum_{|l|\le k}\|D^lf\|_{L^2(K_\theta)}^2<\infty\}\,,\quad k=1,2,\dots\,,
\]
where $K_\theta(\xi):={\rm e}^{\theta\, |\xi|^2/4}$ and $\theta>0$. It is well known (see \cite{EK87,O87}) that the operator
\beq
L_\theta f:=-(\Delta f+\theta\,\frac{\xi}{2}\cdot\nabla f)=-\frac{1}{K_\theta}\nabla \cdot(K_\theta\nabla f)
\label{eq:L}
\eeq
is a positive self adjoint operator on $H^2(K_\theta)=D(L_\theta)$, with eigenvalues given by $\lambda_k=\theta\,\frac{k+1}{2}$, $k\in\N^*$ .
Moreover, the embedding $H^1(K_\theta)\subset L^2(K_\theta)$ is compact and $H^1(K_\theta)\subset L^q(K_{q\,\theta/2})$ for any $q\in[2,\infty)$. In the sequel, we shall use $\theta=1$, $\theta=\e$ and $\theta=\tau:=\min\{1,\e\}$. 

Next, let $S$ be the analytic semigroup generated by $(L_1-I)$ on $L^2(K_1)$, i.e. (\cite{EZ91a}, \cite{O87})
\[
S(s)f(\xi)={\rm e}^s(G({\rm e}^s-1)*f)({\rm e}^{s/2}\xi)\,,\quad s>0\,,\ \xi\in \R^2\,,
\]
and 
$S_{\e}$ be the analytic semigroup generated by $\e^{-1}L_{\e}$ on $L^2(K_{\e})$, i.e. for all $\e>0$
\[
S_{\e}(s)f(\xi)=\left(G\left(\frac{{\rm e}^s-1}{\varepsilon}\right)*f\right)({\rm e}^{s/2}\xi)\,,\quad s>0\,,\ \xi\in \R^2\,.
\]
The following inequalities hold true (see \cite{EK87,O87}), for $s>0$,
\begin{equation}
\|S(s)f\|_{H^1(K_1)}\leq C\,(1+s^{-1/2})\|f\|_{L^2(K_1)}\,,\quad f\in L^2(K_1)\,,
\label{ineq-S*}
\end{equation}
and
\beq
\|S_{\e}(s)f\|_{H^1(K_\tau)}\leq C\,(1+(s/\e)^{-1/2})\|f\|_{L^2(K_\tau)}\,,\quad f\in L^2(K_\tau)\,.
\label{ineq:Se}
\eeq
When $\e\le1$, so that $\tau=\e$, inequality \eqref{ineq:Se} is an immediate consequence of the fact that $L_\e$ is a positive self adjoint operator on $H^2(K_\e)$. When $\e>1$, \eqref{ineq:Se} is not so straightforward. Since the authors do not found any useful references, its proof is given in the Appendix \ref{sec:appendix}. 

With these new semigroups, the integral solution $(\tilde u,\tilde v)$ of \eqref{KS_tilde_u}-\eqref{KS_tilde_v} is given by
\begin{eqnarray}
\label{eq:Duhamel_tilde_u}
&&\tilde u(s+s_0)=S(s)\tilde u(s_0)-
\int_0^s S(s-\sigma)(\nabla\cdot (\tilde u\,\nabla\tilde v(\sigma+s_0)))\,d\sigma\quad\quad \\
\label{eq:Duhamel_tilde_v}
&&\tilde v(s+s_0)=S_{\e}(s)\tilde v(s_0)
+\frac{1}{\e}\int_0^s S_{\e}(s-\sigma)\tilde u(\sigma+s_0)\,d\sigma\,,
\end{eqnarray}
for any $s,s_0\ge0$, and corresponds to the integral solution $(u,v)$ of \eqref{KSu}-\eqref{KSv} through \eqref{eq:tilde_uv}  or \eqref{eq:tilde_uv 2}.  Moreover, for any fixed $s_0>0$, there exists $C=C(s_0)>0$ such that  for any $p\in[1,\infty]$ it holds
\beq
\|\tilde u(s+s_0)\|_{L^{p}(\R^2)}\le C(s_0)^{1-\f1p}\,,\qquad s\ge0\,,
\label{est:suptildeu}
\eeq
and for any $r\in[2,\infty]$ it holds
\beq
\|\nabla\tilde v(s+s_0)\|_{L^r(\R^2)}\leq C(s_0)^{\f12-\f1r}\,, \qquad\|\Delta\tilde v(s+s_0)\|_{L^r(\R^2)}\leq C(s_0)^{1-\f1r}\,,\qquad s\ge0\,.
\label{eq:bo1}
\eeq
These estimates are inherited by the decay properties \eqref{reg_effetct1}-\eqref{reg_effetct2} and \eqref{eq:delta v} on $(u,v)$, and the constant $C(s_0)\sim s_0^{-1}$ as $s_0$ tends to 0. 

The Lemma below will be also used in the sequel. 
\begin{lemma}[\cite{EK87}]\label{lm:EK}
Let $g(\xi):=\theta\f{|\xi|^2}4$, $\theta>0$. The following inequality holds true for all $f\in H^1(K_\theta)$
\beq
\f12\int_{\R^2}|f(\xi)|^2K_\theta(\xi)(\Delta g+\f12|\nabla g|^2)(\xi)\,d\xi\le\int_{\R^2}|\nabla f(\xi)|^2K_\theta(\xi)\,d\xi\,.
\label{eq:moment dans L2}
\eeq
Furthermore, given any $\delta>0$ and $q>2$, there exist $C(\delta,q)>0$ and $R(\delta)>0$ such that  for all $f\in H^1(K_\theta)\cap L^q_{loc}(\R^2)$ it holds
\[
\|f\|^2_{L^2(K_\theta)}\leq \delta\,\|\nabla f\|^2_{L^2(K_\theta)}+C(\delta,q)\|f\|^2_{L^q(B(0,R))}\,.
\label{ineq-H1K}
\]
\end{lemma}

With the help of the functional setting introduced above, we shall prove the following.

\begin{proposition}[Long time behavior I]\label{prop:longtime-alpha=0}
Assume $\e>0$, $\alpha=0$, $u_0\in L^2(K_1)$,  $v_0\in H^1(K_\tau)$, $u_0\ge0$, $v_0\ge0$. Let $(u,v)$ be a non-negative global solution of \eqref{KSu}-\eqref{KSv} such that $u\in L^\infty((0,\infty),L^1(\R^2))$, $v\in L^\infty((0,\infty),\dot{H}^1(\R^2))$ and satisfying estimates \eqref{reg_effetct1}-\eqref{reg_effetct2} and \eqref{eq:grad u}-\eqref{eq:delta v}. Then, if $M=\int_{\R^2}u_0<M^*(\e)$, there exists $t_n\to\infty$ and a self-similar solution $(u_M,v_M)$ s.t. 
\beq\label{eq:long time behavior I}
\lim_{t_n\rightarrow \infty}t_n^{1-\frac{1}{p}}\|u(t_n)-u_M(t_n)\|_{L^p(\R^2)}=0\quad\text{and}\quad
\lim_{t_n\rightarrow \infty}t_n^{\,\frac{1}{2}-\frac{1}{r}}\|\nabla v(t_n)-\nabla v_M(t_n)\|_{L^r(\R^2)}=0\,,
\eeq
for any $p\in[1,\infty]$ and $r\in[2,\infty]$.
\end{proposition}

\begin{theorem}[Long time behavior II]\label{th:longtime-alpha=0}
Assume $\e>0$, $\alpha=0$, $u_0\in L^1(\R^2)$, $v_0\in \dot{H}^1(\R^2)$,
$u_0\ge0$, $v_0\ge0$. Let $(u,v)$ be a non-negative global solution of \eqref{KSu}-\eqref{KSv} such that $u\in L^\infty((0,\infty),L^1(\R^2))$, $v\in L^\infty((0,\infty),\dot{H}^1(\R^2))$ and satisfying estimates \eqref{reg_effetct1}-\eqref{reg_effetct2} and \eqref{eq:grad u}-\eqref{eq:delta v}. Then, if $M=\int_{\R^2}u_0<\widetilde M(\e)$, where $\widetilde M(\e)$ is defined in \eqref{eq:tilde M(e)}, \eqref{lim_uMvM} holds true for any $p\in[1,\infty]$ and $r\in[2,\infty]$.
\end{theorem}


\begin{remark}\label{S4Rem2} 
As we said in the Introduction, the long time behavior of the solutions of \eqref{KSu}-\eqref{KSv} with $\e=1$, has already been studied in \cite{Ferreira}. The following is a simple consequence of that result.  Let $\e=1$ and consider $(u, v)$ and  $(u^*, v^*)$  two global integral solutions. Suppose that $(u^*, v^*)$ is a self-similar solution. Then the asymptotic behaviour
\[
\lim _{ t\to \infty }t^{\left(1-\frac {1} {p}\right)}\|u(t)-u^*(t)\|_{ L^p(\R^2) }=\lim_{ t\to \infty }t^{\frac {1} {2}}\|\nabla v(t)-\nabla v^*(t)\|_{L^\infty(\R^2)}=0
\]
holds for $p\in[r,q]$ and some fixed $r\in(1,2)$, $q\in(2,\infty)$, if and only if
\beq\label{Ferreira10}
\lim _{ t\to \infty }t^{\left(1-\f1p\right)}\|G(t)*(u(0)-u^*(0))\|_{L^p(\R^2)}=
\lim_{t\to \infty }t^{\f12}\|G(t)*(\nabla v(0)-\nabla v^*(0))\|_{L^\infty(\R^2)}=0
\eeq
holds for the same $p\in[r,q]$. 
Condition \eqref{Ferreira10} is fulfilled if, for example $u(0)\ge 0$ is integrable with integral equal to $M$ 
and $u^*(0)=M\delta_0$, 
$v(0)\in\dot{H}^1(\R^2)$ and $v^*(0)=0$
(argue first with $v(0)\in \dot{H}^1(\R^2)\cap W^{1,1}(\R^2)$ and then by density in $\dot{H}^1(\R^2)$).
That is the same initial datum than in our Theorem \ref{th:longtime-alpha=0}.
However, in \cite{Ferreira}, the existence of self-similar solutions of \eqref{KSu}-\eqref{KSv} is proved 
under a smallness condition on both $u_0$ and $v_0$.
\end{remark}

The first step to prove Proposition \ref{prop:longtime-alpha=0} and Theorem \ref{th:longtime-alpha=0} is to show the uniform boundedness of $\|\tilde u(s)\|_{H^1(K_1)}$ and $\|\tilde v(s)\|_{H^2(K_\tau)}$, when the initial data $u_0$ and $v_0$ are taken in $L^2(K_1)$ and $H^1(K_\tau)$ respectively. The compactness of the embedding   $H^1(K_\theta)$ in $L^2(K_\theta)$ shall ensure the relatively compactness of the trajectory. The positivity of $u_0$ and $v_0$ is not required here.
\begin{theorem}\label{S4Thm6}
Assume $\e>0$, $\alpha=0$, $u_0\in L^2(K_1)$, $v_0\in H^1(K_\tau)$. Let $(u,v)$ be a non-negative global solution of \eqref{KSu}-\eqref{KSv} such that $u\in L^\infty((0,\infty),L^1(\R^2))$, $v\in L^\infty((0,\infty),\dot{H}^1(\R^2))$ and satisfying estimates \eqref{reg_effetct1}-\eqref{reg_effetct2} and \eqref{eq:grad u}-\eqref{eq:delta v}. Let $(\tilde u,\tilde v)$ be the corresponding integral solution of \eqref{KS_tilde_u}-\eqref{KS_tilde_v}. Then, $(\tilde u,\tilde v)\in L^\infty([2,\infty); H^1(K_1)\times H^2(K_\tau))$ and the trajectory $(\tilde u(s),\tilde v(s))_{s\ge 2}$ is relatively compact in $L^2(K_1)\times H^1(K_\tau)$. 
\end{theorem}
\begin{proof}
{\sl First step : $\tilde u(s)\in H^1(K_1)$, for $s>0$.} Indeed, using \eqref{eq:Duhamel_tilde_u} and \eqref{ineq-S*} for $s\ge0$ and $s_0>0$, we have
\beqan
\|\tilde u(s+s_0)\|_{H^1(K_1)} 
&\le& \|S(s)\tilde u(s_0)\|_{H^1(K_1)}+\int_0^s \|S(s-\sigma)(\nabla\cdot (\tilde u\nabla\tilde v(\sigma+s_0)))\|_{H^1(K_1)}d\sigma \\
&\le& \|S(s+s_0)u_0\|_{H^1(K_1)}+C\int_0^s (1+(s-\sigma)^{-1/2})\|\nabla\tilde u(\sigma+s_0)\cdot \nabla\tilde v(\sigma+s_0)\|_{L^2(K_1)}\,d\sigma \\
& & +C\int_0^s (1+(s-\sigma)^{-1/2})\|\tilde u(\sigma+s_0)\Delta\tilde v(\sigma+s_0)\|_{L^2(K_1)}\,d\sigma \\
&\le& C(1+(s+s_0)^{-1/2})\|u_0\|_{L^2(K_1)}\\
& & +C\sup_{0\le \sigma\le s}\|\nabla\tilde v(\sigma+s_0)\|_{L^{\infty}(\R^2)}\int_0^s (1+(s-\sigma)^{-1/2})\|\tilde u(\sigma+s_0)\|_{H^1(K_1)}\,d\sigma \\
& & +C\sup_{0\le \sigma\le s}\|\Delta\tilde v(\sigma+s_0)\|_{L^{\infty}(\R^2)}\int_0^s (1+(s-\sigma)^{-1/2})\|\tilde u(\sigma+s_0)\|_{H^1(K_1)}\,d\sigma\,.
\eeqan
Then, from estimates \eqref{eq:bo1} it follows that
\[
\|\tilde u(s+s_0)\|_{H^1(K_1)} \le C(1+(s+s_0)^{-1/2})\|u_0\|_{L^2(K_1)}+
C(s_0)\int_0^s (1+(s-\sigma)^{-1/2})\|\tilde u(\sigma+s_0)\|_{H^1(K_1)}\,d\sigma.
\]
Applying Gronwall's lemma \ref{lem:Gronwall}, for any $T>0$ there exists $C=C(T,s_0,\|u_0\|_{L^2(K_1)})>0$ such that 
\beq
\|\tilde u(s+s_0)\|_{H^1(K_1)} \le C\,(s+s_0)^{-1/2},\qquad s\in [0,T]\,,
\label{ineq:normH1K}
\eeq
i.e. $\tilde u(s)\in H^1(K_1)$, for $s>0$. 

{\sl Second step : $\tilde u\in L^\infty([s_0,\infty)\,;\,L^2(K_1))$, for any $s_0>0$.}
Let $\theta=1$ and $E$ be the eigenspace corresponding to the first eigenvalue, $\lambda_1=1$, of the operator $L_1$ defined in \eqref{eq:L}.  $E$ is spanned by $K_1^{-1}(\xi)$. Let $\varphi(\xi):=c\,K_1^{-1}(\xi)$, where $c:=(\int_{\R^2}K_1^{-1}(\xi)d\xi)^{-1}$ so that $\int_{\R^2}\varphi(\xi)\,d\xi=1$. Since $\tilde u(s)\in L^2(K_1)$, $\tilde u$ may be written~as 
\[
\tilde u(\xi,s)=M\,\varphi(\xi)+w(\xi,s)\,,\qquad s>0\,,
\]
where $M=\int_{\R^2}\tilde u(\xi,s)\,d\xi=\int_{\R^2}u_0(\xi)\,d\xi$ and $w(s)\in E^\bot$ (the orthogonal of $E$ in $L^2(K_1)$) for $s>0$. Moreover, owing to $\int_{\R^2}w(\xi,s)\,d\xi=0$ and $\|\varphi\|_{L^2(K_1)}^2=c$, it holds 
\beq
\|\tilde u(s)\|_{L^2(K_1)}^2=M^2\,c+\|w(s)\|_{L^2(K_1)}^2\,.
\label{ineq-term2}
\eeq
Therefore, it is enough to control the $L^2(K_1)$ norm of $w(s)$ uniformly in time in order to control the $L^2(K_1)$ norm of $\tilde u(s)$ uniformly in time. 

It is easily seen that $w$ satisfies
\[
w_s+L_1w-w=-\nabla\cdot(\tilde u\nabla\tilde v)\,.
\]
Multiplying the above equation by $w\,K_1=\tilde u\,K_1-M\,c$ and integrating over $\R^2$, we obtain
\beq
\begin{split}
\f12\|w(s)\|_{L^2(K_1)}^2&+\|\nabla w(s)\|_{L^2(K_1)}^2-\|w(s)\|_{L^2(K_1)}^2\\
&= -\int_{\R^2}(\nabla\tilde u(\xi,s)\cdot\nabla\tilde v(\xi,s)+\tilde u(\xi,s)\Delta\tilde v(\xi,s))\tilde u(\xi,s)\,K_1(\xi)\,d\xi\\
&\le \|\nabla\tilde v(s)\|_{L^\infty(\R^2)}\|\nabla\tilde u(s)\|_{L^2(K_1)}\|\tilde u(s)\|_{L^2(K_1)}+
\|\Delta\tilde v(s)\|_{L^\infty(\R^2)}\|\tilde u(s)\|_{L^2(K_1)}^2.
\label{est:w}
\end{split}
\eeq
Next, using identities \eqref{ineq-term2} and
\[
\|\nabla\tilde u(s)\|_{L^2(K_1)}^2=M^2\|\nabla\varphi(s)\|_{L^2(K_1)}^2+\|\nabla w(s)\|_{L^2(K_1)}^2=M^2\,C+\|\nabla w(s)\|_{L^2(K_1)}^2\,,
\]
and estimates \eqref{eq:bo1}, \eqref{est:w} becomes for $s\ge s_0>0$
\beq
\begin{split}
\f12\f{d}{ds}\|w(s)\|_{L^2(K_1)}^2&+\|\nabla w(s)\|_{L^2(K_1)}^2-\|w(s)\|_{L^2(K_1)}^2\le C(s_0)\left(M^2\,c+\|w(s)\|_{L^2(K_1)}^2\right)\\
&+ C(s_0)\left(M\,C+\|\nabla w(s)\|_{L^2(K_1)}\right)\left(M\,c^{\f12}+\|w(s)\|_{L^2(K_1)}\right)\,.
\end{split}
\label{ineq-w}
\eeq
As a consequence of \eqref{est:suptildeu}, we are allowed to apply Lemma \ref{lm:EK} to $w(s)$ for $s\ge s_0>0$ to obtain 
\beq
\|w(s)\|^2_{L^2(K_1)}\leq \delta\,\|\nabla w(s)\|^2_{L^2(K_1)}+C(\delta,s_0)\,,
\label{est:delta_w}
\eeq
for any arbitrary $\delta>0$. Using \eqref{est:delta_w} in the r.h.s. of \eqref{ineq-w} many times as needed, we get  
\beq
\f12\f{d}{ds}\|w(s)\|_{L^2(K_1)}^2+\|\nabla w(s)\|_{L^2(K_1)}^2-\|w(s)\|_{L^2(K_1)}^2\le C(\delta,M,s_0)+\delta\,\|\nabla w(s)\|^2_{L^2(K_1)}\,,\quad s\ge s_0>0\,.
\label{ineq-w 2}
\eeq
Finally, since $w(s)\in H^1(K_1)\cap E^\bot$ and that the second eigenvalue of $L_1$ is $\lambda_2=\f32$, we have
\beq
\|\nabla w(s)\|_{L^2(K_1)}^2\ge \f32\,\|w(s)\|^2_{L^2(K_1)}\,,\qquad s>0\,,
\label{ineqlambda2}
\eeq
Owing to \eqref{ineqlambda2},  \eqref{ineq-w 2} becomes
\[
\f12\frac{d}{ds}\|w(s)\|_{L^2(K_1)}^2+\f12(1-3\delta)\|w(s)\|_{L^2(K_1)}^2\le C(\delta,M,s_0), \quad s\ge s_0>0\,.
\]
Choosing $\delta<\f13$ and integrating the above inequality over $(s_0,s)$, we obtain 
\[
\|w(s)\|_{L^2(K_1)}^2\le 2\,C(\delta,M,s_0)\,{\rm e}^{-(1-3\delta)s_0}+\|w(s_0)\|_{L^2(K_1)}^2{\rm e}^{-(1-3\delta)(s-s_0)}, \quad s\ge s_0>0\,,
\]
and the second step is proved.

{\sl Third step : $\tilde u\in L^\infty([2,\infty)\,;\,H^1(K_1))$.} Proceeding as in the first step, from \eqref{eq:Duhamel_tilde_u} we obtain, for $\tau\ge s_0>0$ and $s\ge0$,
\[
\|\tilde u(s+\tau)\|_{H^1(K_1)}\leq C(1+s^{-1/2})\|\tilde u(\tau)\|_{L^2(K_1)}+C(s_0)\int_0^s(1+(s-\sigma)^{-1/2})\|\tilde u(\sigma+\tau)\|_{H^1(K_1)}.
\]
Applying Gronwall's lemma \ref{lem:Gronwall}, for any $T>0$ there exists a constant $C(T,s_0)>0$ such that  
\[
\|\tilde u(s+\tau)\|_{H^1(K_1)}\leq C(T,s_0)\,s^{-\f12}, \qquad s\in (0,T]\,.
\]
It is worth noticing that the constant $C(T,s_0)$ does not depend on $\tau$ owing to the uniform boundedness of $\|u(\tau)\|_{L^2(K_1)}$. Taking $T=2$ and $s=s_0=1$, we have
\[
\|\tilde u(1+\tau)\|_{H^1(K_1)}\leq C, \qquad \tau\ge1\,,
\]
and the proof of this step is complete.

{\sl Fourth step : $\tilde v\in L^\infty([1+s_0,\infty); H^2(K_\tau))$, for any $s_0>0$.}
We shall apply to $\tilde v$ some of the previous arguments. Thus, we shall skip some details. Using \eqref{eq:Duhamel_tilde_v}, \eqref{ineq:Se} and \eqref{ineq:normH1K}  we have for any $T>0$ and $s\in(0,T]$
\beqan
\|\tilde v(s)\|_{H^2(K_\tau)} &\le & \|S_{\e}(s)v_0\|_{H^2(K_\tau)}+\f1\e\int_0^s \|S_{\e}(s-\sigma)\tilde u(\sigma)\|_{H^2(K_\tau)}\,d\sigma \\
& \le &C\,(1+(s/\e)^{-1/2})\|v_0\|_{H^1(K_\tau)}+\e^{-1}\,C(T)\int_0^s(1+(\f{s-\sigma}\e)^{-1/2})\,\sigma^{-1/2}\,d\sigma<\infty\,,
\eeqan
i.e.  $\tilde v(s)\in H^2(K_\tau)$ for $s>0$.

Next, if $0<\e\le1$, multiplying by $\tilde v\, K_\e$ the equation satisfied by $\tilde v$, i.e. $\e\,\tilde v_s+L_\e\, \tilde v=\tilde u$, we obtain, after integration over $\R^2$,
\beq
\begin{split}
\f\e2\frac{d}{ds}\|\tilde v(s)\|_{L^2(K_\e)}^2+\|\nabla\tilde v(s)\|_{L^2(K_\e)}^2 
\le \delta \|\tilde v(s)\|_{L^2(K_\e)}^2+\f1{4\delta}\|\tilde u(s)\|_{L^2(K_1)}^2\,,
\end{split}
\label{ineq:dstilde-v}
\eeq
with $\delta>0$  to be chosen later. Since the first eigenvalue of $L_\e$ defined in \eqref{eq:L} is $\lambda_1=\e$, we have 
\beq
\|\nabla\tilde v(s)\|_{L^2(K_\e)}^2\geq\e\,\|\tilde v(s)\|_{L^2(K_\e)}^2\,,\quad s>0\,.
\label{est:leaste} 
\eeq
Therefore, using \eqref{est:leaste} and the uniform boundedness of $\|\tilde u(s)\|_{L^2(K_1)}^2$ for $s\ge s_0>0$ previously obtained, \eqref{ineq:dstilde-v} becomes
\[
\f\e2\frac{d}{ds}\|\tilde v(s)\|_{L^2(K_\e)}^2+(\e-\delta)\|\tilde v(s)\|_{L^2(K_\e)}^2 \le C(\delta,s_0)\,,\quad s\ge s_0>0\,.
\]
Integrating the above differential inequality (with $\delta<\e$) over $(s_0,s)$, we get that $\tilde v\in L^{\infty}([s_0,\infty),L^2(K_\e))$. If $\e>1$, multiplying \eqref{KS_tilde_v} by $\tilde v\, K_1$ and integrating over $\R^2$, we obtain 
\[
\begin{split}
\f\e2\frac{d}{ds}\|\tilde v(s)\|_{L^2(K_1)}^2+\|\nabla\tilde v(s)\|_{L^2(K_1)}^2 
&\le\f{\e-1}2\int_{\R^2}\nabla\tilde v^2(\xi,s)\cdot\nabla K_1(\xi)d\xi+ \delta \|\tilde v(s)\|_{L^2(K_1)}^2+\f1{4\delta}\|\tilde u(s)\|_{L^2(K_1)}^2\\
&\le\left(\delta-\f{\e-1}2\right)\|\tilde v(s)\|_{L^2(K_1)}^2+\f1{4\delta}\|\tilde u(s)\|_{L^2(K_1)}^2\,.
\end{split}
\]
Choosing $\delta<\f{\e-1}2$ and integrating over $(s_0,s)$, we get that $\tilde v\in L^{\infty}([s_0,\infty),L^2(K_1))$. 

Furthermore, proceeding as before for $\nu\ge s_0>0$ and $s\ge0$, we have
\[
\begin{split}
\|\tilde v(s+\nu)\|_{H^1(K_\tau)} &\le \|S_{\e}(s)\tilde v(\nu)\|_{H^1(K_\tau)}+\frac{1}{\e}\int_0^s \|S_{\e}(s-\sigma)\tilde u(\sigma+\nu)\|_{H^1(K_\tau)}\,d\sigma \\
& \le  C(1+(s/\e)^{-1/2})\|\tilde v(\nu)\|_{L^2(K_\tau)}+\f C\e\int_0^s (1+(\f{s-\sigma}\e)^{-1/2})\|\tilde u(\sigma+\nu)\|_{L^2(K_1)}\,d\sigma\,.
\end{split}
\]
Choosing $s=1$ and owing to the uniform boundedness of $\|\tilde v(\nu)\|_{L^2(K_\tau)}$ and $\|\tilde u(\sigma+\nu)\|_{L^2(K_1)}$, we have for all $\nu\ge s_0$
\[
\|\tilde v(1+\nu)\|_{H^1(K_\tau)}\le C(\e,s_0)\,.
\]
Repeating the same argument for the $H^2(K_\tau)$ norm, we arrive to $\tilde v\in L^\infty([1+s_0,\infty); H^2(K_\tau))$.
\end{proof}

\begin{remark}
\label{S4R1}
Notice that in  Theorem \ref{S4Thm6}, we ask the initial data $(u_0, v_0)$ to belong to $L^2(K_1)\times H^1(K_\tau )$, where $\tau=\min\{1,\e\}$,  instead of $L^2(K_1)\times H^1(K_1)$. The reason is that we do not know if, when $\varepsilon <1$ and $f\in L^2(K_1)$,  $S _{ \varepsilon }f\in L^2(K_1)$, and therefore, we can not consider the two functions $u(t)$ and $v(t)$ in the functional spaces with the same weight.
\end{remark}

With the above compactness result, we are finally able to prove Proposition \ref{prop:longtime-alpha=0} and Theorem \ref{th:longtime-alpha=0}.

\begin{proof} [Proof of Proposition \ref{prop:longtime-alpha=0}]
Let $(\tilde u,\tilde v)$ be the integral solution of \eqref{KS_tilde_u}-\eqref{KS_tilde_v} corresponding to $(u,v)$. 
Let 
\[
\label{omegalimitset}
\begin{split}
\omega(u_0,v_{0}):=\{(f,g)\in L^2(K_1)\times H^1(K_{\tau}):\ &\exists\ s_{n}\rightarrow \infty 
\mbox{ such that } \; \tilde u(s_{n})\rightarrow f \mbox{ in } L^2(K_1)\\ \text{and }
&|\nabla\tilde v(s_{n})|\rightarrow |\nabla g| \mbox{ in } L^2(K_{\tau})\ \text{as }\ n\to\infty\}\,.
\end{split}
\]
be the $\omega$-limit set of $(\tilde u,\tilde v)$. This set is non empty due to the relatively compactness of the trajectory $(\tilde u(s),\tilde v(s))_{s\ge 2}$ proved in Theorem \ref{S4Thm6}. Moreover, since the embedding $L^2(K_1)\subset L^1(\R^2)$ is continuous, for each $(f,g)\in\omega(u_0,v_{0})$, we have
\[
\int_{\R^2}f(\xi)\, d\xi=\int_{\R^2}\tilde u(\xi,s)\,d\xi=\int_{\R^2}u_0(\xi)\,d\xi=M\,, \quad\, s>0\,.
\label{eq:massconserved}
\]

Next, let us rewrite the Liapunov functional~\eqref{eq:fe} in the $(\tilde u,\tilde v)$ variables
\beq
\tilde{\mathcal E}(s)=-M\,s+ \int_{\R^2} (\tilde u\log\tilde u)(\xi,s)\,d\xi - \int_{\R^2}(\tilde u\,\tilde v)(\xi,s)\,d\xi + \f12\int_{\R^2} |\nabla\tilde v(\xi,s)|^2\,d\xi\,.
\label{eq:tildeE}
\eeq
It is worth noticing that \eqref{eq:tildeE} makes sense because $(\tilde u(s),\tilde v(s))\in H^1(K_1)\times H^2(K_\tau)$, for $s>0$, and the solution is non-negative. 
Furthermore, $\tilde{\mathcal E}(s)$ is strictly decreasing along the trajectories, since it satisfies 
\beq
\f d{ds}\tilde{\mathcal E}(s)=- \int_{\R^2}\tilde u(\xi,s)|\nabla(\log\tilde u-\tilde v)(\xi,s)|^2\,d\xi -\e \int_{\R^2}(\partial_s\tilde v-\f\xi2\cdot\nabla\tilde v)^2(\xi,s)\,d\xi
\label{eq:tildeE'}
\eeq
and the first integral term in the r.h.s. of \eqref{eq:tildeE'} can be identically zero iff  $\tilde u=C\,e^{\tilde v}$, which is not possible for integrability reasons. Then, applying the classical LaSalle invariance principle, we have that $\omega(u_0,v_{0})$ is contained in the set of the stationary solutions of 
\eqref{KS_tilde_u}-\eqref{KS_tilde_v}, or equivalently in the set of solutions of the elliptic system \eqref{KSuSS}-\eqref{KSvSS}. Consequently, there exists $t_n\to\infty$ and a self-similar solution $(u_M,v_M)$ s.t. (after undoing the change of variable)
\[
\lim_{t_n\rightarrow \infty}\|u(t_n)-u_M(t_n+1)\|_{L^1(\R^2)}=
\lim_{t_n\rightarrow \infty}\|\nabla v(t_n)-\nabla v_M(t_n+1)\|_{L^2(\R^2)}=0\,.
\]
Since by the Lebesgue's Dominated Convergence Theorem it holds
\[
\lim_{t_n\to\infty}\left\|u_M(t_n)-u_{M}(t_n+1)\right\|_{L^1(\R^2)}
=\lim_{t_n\to\infty}\|\nabla v_M(t_n)-\nabla v_M(t_n+1)\|_{L^2(\R^2)}=0\,,
\]
\eqref{eq:long time behavior I} is proved for $p=1$ and $r=2$.

Next, due to the uniform boundedness of $(\tilde u(s),|\nabla\tilde v(s)|)_{s\ge s_0>0}$ in $L^p(\R^2)\times L^r(\R^2)$, $p\in[1,\infty]$, $r\in[2,\infty]$  (see \eqref{est:suptildeu}-\eqref{eq:bo1}), the claim \eqref{eq:long time behavior I} for $p\in(1,\infty)$ and $r\in(2,\infty)$ follows by interpolation.

Finally, in the case $p=r=\infty$, \eqref{eq:long time behavior I} follows from the previous results. Indeed, the Gagliardo-Nirenberg inequality for $q\in(2,\infty)$ and $a=\f2q$ gives us
\[
\begin{split}
t\,\|u(t)-& u_M(t)\|_{L^\infty(\R^2)}\le C\,t\|u(t)-u_{M}(t)\|_{L^q(\R^2)}^{1-a}\|\nabla u(t)-\nabla u_M(t)\|_{L^q(\R^2)}^a\\
&\le C\left(t^{1-\f1q}\left\|u(t)-u_{M}(t)\right\|_{L^q(\R^2)}\right)^{1-a}
\left(t^{\frac{3}{2}-\frac{1}{q}}\|\nabla u(t)\|_{L^q(\R^2)}+t^{\frac{3}{2}-\frac{1}{q}}\|\nabla u_M(t)\|_{L^q(\R^2)}\right)^a
\end{split}
\]
and
\[
\begin{split}
t^{\f12}\|\nabla v(t)-&\nabla v_M(t)\|_{L^\infty(\R^2)} \le
C\,t^{\f12}\|\nabla v(t)-\nabla v_M(t)\|_{L^{q}(\R^2)}^{1-a}\|\Delta v(t)-\Delta v_M(t)\|_{L^{q}(\R^2)}^a\\
& \le C \left(t^{\f12-\f1q}\|\nabla v(t)-\nabla v_M(t)\|_{L^{q}(\R^2)}\right)^{1-a}
\left(t^{1-\frac{1}{q}}\|\Delta v(t)\|_{L^{q}(\R^2)}+t^{1-\frac{1}{q}}\|\Delta v_M(t)\|_{L^{q}(\R^2)}\right)^a\,.
\end{split}
\]
Moreover, by the definition \eqref{phi} of $\phi$, equation \eqref{phi2} and estimates \eqref{upperS}-\eqref{est:phi'infty}, we get for all $t>0$
\[
t^{\f32-\f1q}\|\nabla u_M(t)\|_{L^q(\R^2)}=\|\nabla U_M\|_{L^q(\R^2)}=C(q)\,\|y^{\f12}\phi''(y)\|_{L^q(0,\infty)}=C'(q)\,.
\]
Similarly, by the definition \eqref{psi} of $\psi$, equation \eqref{S} and estimates \eqref{upperS}-\eqref{est:phi'infty}, we have
\[
t^{1-\frac{1}{q}}\|\Delta v_M(t)\|_{L^{q}(\R^2)}=\|\Delta V_M\|_{L^{q}(\R^2)}=C(q)\|V''_M(\sqrt y)+y^{-\f12}V'_M(\sqrt y)\|_{L^q(0,\infty)}
=C'(q)\,.
\]
Then, using the above computations and estimates \eqref{eq:grad u}-\eqref{eq:delta v}, we obtain the proof.
\end{proof}

\begin{proof} [Proof of Theorem \ref{th:longtime-alpha=0}]
Let $u_{0,n}\in L^2(K_1)$ and $v_{0,n}\in H^1(K_\tau)$ be non-negative sequences such that  
\[
u_{0,n}\rightarrow u_0 \;\; \mbox{ in } \;\; L^1(\R^2) \;\; \mbox{ and } \;\;
|\nabla v_{0,n}|\rightarrow |\nabla v_0|\;\; \mbox{ in } \;\; L^2(\R^2)
\;\; \mbox{ as } \;\; n\rightarrow \infty.
\]
Then,
\[\label{eq:MntoM}
M_n:=\int_{\R^2}u_{0,n}(x)\,dx\to M\,,\qquad\text{as } n\to\infty\,.
\]
Let $(u_n,v_n)$ be the non-negative global solution of \eqref{KSu}-\eqref{KSv} with initial data $(u_{0,n},v_{0,n})$ given by Theorem~\ref{th:existence critical} and $(\tilde u_n,\tilde v_n)$ the corresponding integral solution of \eqref{KS_tilde_u}-\eqref{KS_tilde_v}. Let $n$ be sufficiently large, so that $M_n<\widetilde M(\e)$. By the inclusion of $\omega(u_{0,n},v_{0,n})$ in the set of the equilibrium states proved in Proposition \ref{prop:longtime-alpha=0} and due to the uniqueness of the equilibrium given by Theorem~\ref{th:uniqueness_self-similar}, we have : $\omega(u_{0,n},v_{0,n})=\{(U_{M_n},V_{M_n})\}$, where $(U_{M_n},U_{M_n})$ is the unique solution of \eqref{KSuSS}-\eqref{KSvSS} corresponding to~$M_n$. Hence, for any fixed $n$ (sufficiently large), 
\[
\tilde u_n(s)\to U_{M_n}\quad\text{in } L^2(K_1)\quad\text{and}\quad
|\nabla\tilde v_n(s)|\to|\nabla V_{M_n}|\quad\text{in }L^2(K_\tau)\,,\quad\text{as }s\to\infty\,.
\]
The limits above hold true also in $L^1(\R^2)$ and $L^2(\R^2)$ respectively. Moreover, owing to the uniform boundedness of $(\tilde u_n(s),|\nabla\tilde v_n(s)|)_{s\ge s_0>0}$ in $L^p(\R^2)\times L^r(\R^2)$, $p\in[1,\infty]$, $r\in[2,\infty]$  (see \eqref{est:suptildeu}-\eqref{eq:bo1}) and the boundedness of $(U_{M_n},|\nabla V_{M_n}|)$ (see \eqref{eq:UMn bound}-\eqref{eq:gradVMn bound}), for $n$ fixed, we easily obtain by interpolation that
\beq
\tilde u_n(s)\to U_{M_n}\quad\text{in } L^p(\R^2)\quad\text{and}\quad
|\nabla\tilde v_n(s)|\to|\nabla V_{M_n}|\quad\text{in }L^r(\R^2)\,,\quad\text{as }s\to\infty\,,
\label{eq:conv}
\eeq
for every $p\in [1,\infty)$, $r\in [2,\infty)$ and any fixed $n$ (sufficiently large).

On the other hand, let $(u_M,v_M)$ be the unique self-similar solution corresponding to $M$ according to Theorem \ref{th:uniqueness_self-similar}, with profile $(U_M,V_M)$. Then, 
\beq\label{ineq:u-uM}
t^{1-\f1p}\left\|u(t)-u_{M}(t)\right\|_{L^p(\R^2)}\le t^{1-\f1p}\|u(t)-u_n(t)\|_{L^p(\R^2)}+t^{1-\f1p}\left\|u_n(t)-u_{M_n}(t)\right\|_{L^p(\R^2)}
+\|U_{M_n}-U_{M}\|_{L^p(\R^2)}
\eeq
and 
\beq\label{ineq:v vM}
\begin{split}
t^{\f12-\f1r}\|\nabla v(t)-\nabla v_M(t)\|_{L^r(\R^2)}\le t^{\f12-\f1r}\|\nabla v(t)-\nabla v_n(t)\|_{L^r(\R^2)}&+
t^{\f12-\f1r}\|\nabla v_n(t)-\nabla v_{M_n}(t)\|_{L^r(\R^2)}\\
&+\|\nabla V_{M_n}-\nabla V_M\|_{L^r(\R^2)}\,.
\end{split}
\eeq
From the continuity property \eqref{in:continuity}, the first terms in the r.h.s. of \eqref{ineq:u-uM} and \eqref{ineq:v vM} tend to $0$ as $n\to\infty$. 
From \eqref{eq:conv}, undoing the change of variables and proceeding as in Proposition \ref{prop:longtime-alpha=0}, the second terms in the r.h.s. of \eqref{ineq:u-uM} and \eqref{ineq:v vM} tend to $0$ as $t\to\infty$. Furthermore, from Proposition \ref{pr:continuity} the third terms in the r.h.s. of \eqref{ineq:u-uM} and \eqref{ineq:v vM} tend to 0 as $n\to\infty$. Therefore, \eqref{lim_uMvM} is proved for any $p\in[1,\infty)$ and $r\in[2,\infty)$.

Finally, for $p=r=\infty$, \eqref{lim_uMvM} follows from the Gagliardo-Nirenberg inequality as in Proposition~\ref{prop:longtime-alpha=0}.
\end{proof}
\section{Long time behavior : the case $\alpha>0$}
\label{sec:alphapos}
In the case $\alpha>0$, system \eqref{KSu}-\eqref{KSv} is no more invariant under the space-time scaling \eqref{eq:scaling}. Therefore, self-similar solutions do not exist. However, in that case one can take advantage of the degradation term in equation \eqref{KSv}, giving an improved time decay of $|\nabla v|$, to prove that global solutions $(u,v)$ of \eqref{KSu}-\eqref{KSv} behave, in the first component, as the heat kernel $G(t)$ as~$t\to\infty$. In other words, the non-local chemotactic term in equation \eqref{KSu} is too week and system \eqref{KSu}-\eqref{KSv} is weakly non-linear. 

\begin{theorem}\label{th:asymptbeh-alphapos}
Assume $\e>0$, $\alpha>0$, $u_0\in L^1(\R^2)$, $|\nabla v_0|\in L^2(\R^2)$ and $M=\int_{\R^2}u_0(x)\,dx$.  Let $(u,v)$ be a global solution of \eqref{KSu}-\eqref{KSv} such that $u\in L^\infty((0,\infty),L^1(\R^2))$, $v\in L^\infty((0,\infty),\dot{H}^1(\R^2))$ and satisfying estimates \eqref{reg_effetct1}-\eqref{reg_effetct2}.
Then, the solution $(u,v)$ verifies 
\begin{enumerate}
\item for all $p\in [1,\infty]$, 
\beq
t^{1-\f 1p}\|u(t)-MG(t)\|_{L^p(\R^2)}\longrightarrow 0 \quad \mbox{ as }
\quad t\rightarrow \infty\,;
\label{lim-uMG}
\eeq
\item for all $r\in [2,\infty)$ and $q\in \,(\f{2r}{r+2},\infty)$ or $r=\infty$ and $q>2$,
there exists $C=C(r,q,\e,\alpha)$ such that 
\beq
t^{\f12-\f 1{r}}\norm{\nabla v(t)}_{L^r(\R^2)}\le C\,t^{-(\frac{1}{r}-\frac{1}{q}+\frac{1}{2})}\,,\quad t>0\,.
\label{lim-nabla-v}
\eeq
\end{enumerate}
\end{theorem}
\begin{proof}
Let $p\in [1,\infty]$. From \eqref{eq:Duhamel_u}, we have
\beq
\|u(t)-MG(t)\|_{L^p(\R^2)}\le \|G(t)*u_0-MG(t)\|_{L^p(\R^2)}+\sum_{i}\int_0^t \|\partial_{i} G(t-s)*(u(s)\partial_{i} v(s))\|_{L^p(\R^2)}\,ds\,.
\label{ineq-uM0t}
\eeq
Since $\lim_{t\to0}t^{1-\f1p}\|G(t)*u_0-MG(t)\|_{L^p(\R^2)}=0$, we need to estimate only the second term on the right hand side of \eqref{ineq-uM0t} to obtain \eqref{lim-uMG}. To begin with, we shall derive the improved decay rate \eqref{lim-nabla-v} of $\nabla v(t)$. 

Let $r\in [2,\infty]$ and $q>2$ such that $\f1q-\f1r<\f12$. From \eqref{eq:Duhamel_v} and \eqref{reg_effetct1}, there exist $C_1=C_1(r,q)>0$ and $C_2=C_2(\e,r,q,\|u_0\|_{L^1(\R^2)})>0$ such that  for $t>0$ it holds 
\beq
\begin{split}
\|\nabla v(t)\|_{L^r(\R^2)}&\le C_1\,e^{-\f{\alpha}{\e}t}\,t^{-(\f12-\f1r)}\|\nabla v_0\|_{L^2(\R^2)}+C_2\int_0^t\f{e^{-\frac{\alpha}{\e}(t-s)}}
{(t-s)^{\f1q-\f1r+\f12}}\,\f1{s^{1-\f1q}}\,ds\,.
\end{split}
\label{est:grad v-alpha}
\eeq
We have on the one hand 
\[
\int_0^{\f t2}\f{e^{-\frac{\alpha}{Öe}(t-s)}}{(t-s)^{\f1q-\f1r+\f12}}\,\f1{s^{1-\f1q}}\,ds\le\f{e^{-\frac{\alpha}{\e}\f t2}}{(t/2)^{\f1q-\f1r+\f12}}\,\int_0^{\f t2}\f1{s^{1-\f1q}}\,ds= C(r,q)\,e^{-\frac{\alpha}{\e}\frac{t}{2}}\,t^{-(\f 12-\f 1r)}
\]
and on the other hand 
\[
\int_{\f t2}^t\f{e^{-\frac{\alpha}{\e}(t-s)}}{(t-s)^{\f1q-\f1r+\f12}}\,\f1{s^{1-\f1q}}\,ds
\le (t/2)^{-(1-\f 1q)}\int_0^\infty\f{e^{-\f\alpha\e\tau}}{\tau^{\f1q-\f1r+\f12}}\,d\tau
=C(r,q,\e,\alpha)\,t^{-(1-\f 1q)}\,.
\]
Adding the above two estimates and plugging the resulting one into \eqref{est:grad v-alpha}, we obtain for $t>0$
\[
t^{\f12-\f 1{r}}\norm{\nabla v(t)}_{L^r(\R^2)}\le C(r,q,\|\nabla v_0\|_{L^2(\R^2)})\,e^{-\frac{\alpha}{\e}\frac{t}{2}}+
C(r,q,\e,\alpha)\,t^{-(\f 1r-\f 1q+\f 12)}\,,
\]
implying the existence of $t_0=t_0(r,q,\e,\alpha)$ such that  \eqref{lim-nabla-v} is satisfied for $t\ge t_0$. Since $t^{\f12-\f 1{r}}\norm{\nabla v(t)}_{L^r(\R^2)}$ is bounded for all $t>0$, \eqref{lim-nabla-v} is also satisfied for $t\in(0,t_0)$ with an appropriate constant $C=C(r,q,\e,\alpha)$.

Next, using \eqref{reg_effetct1} and \eqref{lim-nabla-v} with $\f{2r}{r+2}<q<r$, we have 
\beq
\begin{split}
\int_0^{\f t2} \|\partial_{i} G(t-s)*(u(s)\partial_{i} v(s))\|_{L^p(\R^2)}\,ds
&\le \int_0^{\f t2} \|\partial_{i} G(t-s)\|_{L^p(\R^2)}\|u(s)\|_{L^{r'}(\R^2)}\|\nabla v(s)\|_{L^r(\R^2)}\,ds \\
&\le C\f1{(t/2)^{1-\f 1p+\f 12}}\int_0^{\f t2}\f1{s^{\f1r}}\f1{s^{1-\f 1q}}\,ds=C\,t^{-(1-\f1p)-(\f12+\f1r-\f1q)}\,.
\end{split}
\label{ineq-uM0t2}
\eeq
On the other hand, using \eqref{reg_effetct1} and \eqref{lim-nabla-v} with $r>2$
\beq
\begin{split}
\int_{\f t2}^t \|\partial_{i} G(t-s)*(u(s)\partial_{i} v(s))\|_{L^p(\R^2)}\,ds&\le 
\int_{\f t2}^t\|\partial_{i} G(t-s)\|_{L^{r'}(\R^2)}\|u(s)\|_{L^{p}(\R^2)}\|\nabla v(s)\|_{L^r(\R^2)}\,ds\\
&\le C\int_{\f t2}^t\f1{(t-s)^{\f1r+\f12}}\f1{s^{1-\f1p}}\f1{s^{1-\f1q}}\,ds=C\,t^{-(1-\f1p)-(\f12+\f1r-\f1q)}\,.
\end{split}
\label{ineq-uMt2t}
\eeq
Adding \eqref{ineq-uM0t2} and \eqref{ineq-uMt2t} and plugging the resulting estimate into \eqref{ineq-uM0t} we get \eqref{lim-uMG}.
\end{proof}
\begin{remark}
It is worth noticing that the constant $C$ in \eqref{lim-nabla-v} goes to $+\infty$ as $q\to\infty$. Moreover, the asymptotic behavior \eqref{lim-uMG} can be improved if the initial datum $u_0$ is smoother. For instance, if $u_0\in L^1(\R^2,1+|x|)$, it is easy to show from the previous arguments that for all 
$p\in[1,\infty]$ and $\delta\in(0,\f12)$, it holds
\[
\lim_{t\to\infty}t^{(1-\f 1p)+\delta}\|u(t)-MG(t)\|_{L^p(\R^2)}=0\,,
\]
(see also \cite{EZ91a,N06}).
For $\alpha=\e=1$, \eqref{lim-uMG} is established with $p=\infty$ in \cite{N01} assuming 
$(u_0,v_0)\in L^1\cap L^\infty$ nonnegative and $M<4\pi$. 
Notice that in our result, the positivity of the initial data is not required neither is the smallness of $M$.

\end{remark}
\section{Appendix}
\label{sec:appendix}
\begin{proof}[Proof of \eqref{ineq:Se}]
We shall assume first that $f\in L^2(K _\e)$. Then, $v(s):=S _\e(s)f\in H^2(K _\e)$ for $s>0$, and by \eqref{eq:moment dans L2} we deduce that $\nabla\cdot(K_\e\nabla v)\,v\,K_{\tau-\e}\in L^1(\R^2)$ since 
\[
\nabla\cdot(K_\e\nabla v)\,v\,K_{\tau-\e}=\e(\f\xi2\cdot\nabla v\,K _{ \tau /2})\,v\,K _{ \tau /2}+(\Delta vK _{ \tau/2 })\,vK _{ \tau/2 }\,.
\]
Let now $\varphi _n$ be a sequence of regular compactly supported functions converging to one as $n\to \infty$. Multiplying the equation for $v$, i.e. $\e\,v_s+L_\e v=0$, by $v\,K _\tau\,\varphi _n$, integrating over $\R^2$ and using the integration by parts, we get (since $\varphi _n$ is compactly supported)
\[
\begin{split}
\f\e2\f d{ds}\|v(s)\sqrt{\varphi _n}\|^2_{L^2(K_\tau)}&=\int_{\R^2}\nabla\cdot(K_\e\nabla v)v\,K_{\tau-\e}\,\varphi _n \\
&=-\int_{\R^2}|\nabla v|^2K_\tau\, \varphi _n-(1-\f\e\tau)\int_{\R^2}\nabla\left(\f{v^2}2\right)\cdot\nabla K_\tau\, \varphi _n
 -\int_{\R^2}v\,\nabla v\cdot\nabla  \varphi _n\,K_{\tau}\\
&=-\int_{\R^2}|\nabla v|^2K_\tau\, \varphi _n+\f12(1-\f\e\tau)\int_{\R^2}v^2\,\Delta K_\tau \varphi _n
+\f12(1-\f\e\tau)\int  _{ \R^2 }v^2 \nabla K _{ \tau  }\cdot\nabla \varphi _n\\
&\qquad -\int_{\R^2}v\,\nabla v\cdot\nabla\varphi _n\,K_{\tau} \\
&=-\int_{\R^2}|\nabla v|^2K_\tau\, \varphi _n+\f\tau8(\tau-\e)\int_{\R^2}v^2|\xi|^2K_\tau\,\varphi _n+\f12(\tau-\e)\int_{\R^2}v^2K_\tau\, \varphi _n\\
&\qquad +\f12(1-\f\e\tau)\int  _{ \R^2 }v^2\, \nabla K _{ \tau  }\cdot\nabla \varphi _n -\int_{\R^2}v\,\nabla v\cdot\nabla\varphi _n\,K_{\tau} \, .
\end{split}
\]
Since $v(s)\in H^2(K _{ \e })$ and $\tau \le \varepsilon$  we have that $v,|\nabla v|\in L^2(K _{ \tau  })$ and by \eqref{eq:moment dans L2}, $v|\xi|\in L^2(K _{ \tau  })$. Then, taking the limit as $n\to \infty$, and using the  Lebesgue's dominated convergence Theorem we deduce
\[
\begin{split}
\f\e2\f d{ds}\|v(s)\|^2_{L^2(K_\tau)}&=-\int_{\R^2}|\nabla v|^2K_\tau+\f\tau8(\tau-\e)\int_{\R^2}v^2|\xi|^2K_\tau(\xi)+\f12(\tau-\e)\int_{\R^2}v^2K_\tau(\xi)\,,
\end{split}
\]
i.e. $\|v(s)\|^2_{L^2(K_\tau)}$ is exponentially decreasing and
\beq
\int_0^s\|\nabla v(\sigma)\|^2_{L^2(K_\tau)}d\sigma\le\f\e2\|f\|^2_{L^2(K_\tau)}\,,\qquad s>0\,.
\label{eq:est_Ktau_0.Bis}
\eeq

Next, multiplying the equation for $v$ by $-\nabla\cdot( \varphi _n K_\tau\nabla v)$, we get
\beq
\begin{split}
\f\e2\f d{ds}\|\sqrt{\varphi _n}\,\nabla v(s)\|^2_{L^2(K_\tau)}&=-\int_{\R^2}\left[K _{ -\e }\nabla \cdot (K _{ \e }\nabla v)\right][\nabla\cdot(\varphi _n K_\tau\nabla v)]\\
&=-\int_{\R^2}[\Delta v+\frac {\e} {2}\xi\cdot \nabla v][K _{ \tau }\varphi _n\Delta v+\frac {\tau } {2}\xi \cdot \nabla v K _{ \tau  }\varphi _n+\nabla v\cdot \nabla \varphi _n K _{ \tau  }]\\
&=-\int_{\R^2}|\Delta v|^2K_\tau \varphi _n-\f{\e\tau}4\int_{\R^2}(\nabla v\cdot\xi)^2K_\tau \varphi _n
-\f12(\e+\tau)\int_{\R^2}\Delta v(\xi\cdot\nabla v)K_\tau \varphi _n\\
&\qquad-\int_{\R^2}[\Delta v+\frac {\e} {2}\xi\cdot \nabla v][\nabla v\cdot \nabla \varphi _n K _{ \tau  }]\,.
\label{eq:est_Ktau_1.Bis}
\end{split}
\eeq
We have also
\beq
\begin{split}
\int_{\R^2}\Delta v(\xi\cdot\nabla v)K_\tau \varphi _n&=-\int_{\R^2}|\nabla v|^2K_\tau \varphi _n-\f12\int_{\R^2}\xi\cdot\nabla(|\nabla v|^2)K_\tau \varphi _n-\f\tau2\int_{\R^2}(\xi\cdot\nabla v)^2K_\tau \varphi _n-\\
&\qquad - \int_{\R^2}(\xi\cdot\nabla v)\nabla v\cdot \nabla\varphi _n\,K_\tau\\
&=-\f12\int_{\R^2}\nabla\cdot(\xi|\nabla v|^2)K_\tau \varphi _n-\f\tau2\int_{\R^2}(\xi\cdot\nabla v)^2K_\tau \varphi _n
- \int_{\R^2}(\xi\cdot\nabla v)\nabla v\cdot \nabla\varphi _n\,K_\tau\\
&=\f\tau4\int_{\R^2}|\xi|^2|\nabla v|^2K_\tau \varphi _n+\f12\int_{\R^2}|\nabla v|^2(\xi\cdot\nabla \varphi _n)\,K_\tau
-\f\tau2\int_{\R^2}(\xi\cdot\nabla v)^2K_\tau \varphi _n-\\
&\qquad- \int_{\R^2}(\xi\cdot\nabla v)\nabla v\cdot \nabla\varphi _n\,K_\tau\,.
\end{split}
\label{eq:est_Ktau_2.Bis}
\eeq
Therefore, plugging (\ref{eq:est_Ktau_2.Bis}) into (\ref{eq:est_Ktau_1.Bis}),
\[
\begin{split}
\f\e2\f d{ds}\|\sqrt{\varphi _n}\,\nabla v(s)\|^2_{L^2(K_\tau)}&=-\int_{\R^2}|\Delta v|^2K_\tau \varphi _n+\f{\tau^2}4\int_{\R^2}(\nabla v\cdot\xi)^2K_\tau \varphi _n
-\f\tau 8 (\e+\tau)\int_{\R^2}|\xi|^2|\nabla v|^2K_\tau \varphi _n\\
&\quad-\f12(\e+\tau )\left(\f12\int_{\R^2}|\nabla v|^2(\xi\cdot \nabla \varphi _n)K_\tau  - \int_{\R^2}(\xi\cdot\nabla v)\nabla v\cdot \nabla\varphi _n\,K_\tau\right)\\
&\quad-\int_{\R^2}[\Delta v+\frac {\e} {2}\xi\cdot \nabla v][\nabla v\cdot \nabla \varphi _n K _{ \tau  }]\,.
\end{split}
\]
Again, by the Lebesgue's convergence theorem, we deduce
\[
\f\e2\f d{ds}\|\nabla v(s)\|^2_{L^2(K_\tau)}=-\int_{\R^2}|\Delta v|^2K_\tau +\f{\tau^2}4\int_{\R^2}(\nabla v\cdot\xi)^2K_\tau 
-\f\tau 8 (\e+\tau)\int_{\R^2}|\xi|^2|\nabla v|^2K_\tau\le0
\label{eq:est_Ktau_1.BisBis}
\]
i.e. $\|\nabla v(s)\|^2_{L^2(K_\tau)}$ is decreasing, so that \eqref{eq:est_Ktau_0.Bis} implies
\[
s\|\nabla v(s)\|^2_{L^2(K_\tau)}\le\f\e2\|f\|^2_{L^2(K_\tau)}\,,\qquad s>0\,.
\label{eq:est_Ktau_1.BisTer}
\]

Consider now $f\in L^2(K _\tau)$ and a sequence $f_n\in H^1(K _{ \e })$ such that $f_n\to f$ in $L^2(K _{ \tau  })$ as $n\to \infty$, and the corresponding sequence of functions $v_n(s)=S _{ \varepsilon  }(s)f _{ n }$. Since the previous argument may be applied to $(v_n-v_m)$, we have that $\|v_n(s)-v_m(s)\|^2_{L^2(K_\tau)}$ is exponentially decreasing with
\beqan
\|v_n(s)-v_m(s)\|^2_{L^2(K_\tau)}\le \|v_n(0)-v_m(0)\|^2_{L^2(K_\tau)}=\|f_n-f_m\|^2_{L^2(K_\tau)}\,,\quad s>0\,,
\eeqan
and
\[
s\|\nabla (v_n(s)-v_m(s))\|^2_{L^2(K_\tau)}\le \f\e2\| f_n-f_m\|^2_{L^2(K_\tau)}\,,\qquad s>0\,.
\]
Hence, $v_n(s)$ is a Cauchy sequence in $H^1(K _{ \tau  })$, for all $s>0$, $v_n(s)\to v(s)$ in $H^1(K _{ \tau  })$ as $n\to\infty$, $v(s)=S _\e(s)f$ and \eqref{ineq:Se} holds true for $v$.  
%
%
\end{proof}
\begin{acknowledgment}
The first author acknowledges the support of the french ``ANR blanche'' project Kibord : ANR-13-BS01-0004.
The second author is supported by DGES Grant 2011-29306-C02-00 and Basque Government Grant IT641-13.
\end{acknowledgment}
%
%

\noindent ${^a}$ Laboratoire d'Analyse et Probabilit\'e \\
Universit\'e d'Evry Val d'Essonne, \\
23 Bd. de France, F--91037 Evry Cedex \\
emails: lucilla.corrias@univ-evry.fr, julia.matos@univ-evry.fr
\bigskip

\noindent ${^b}$ Departamento de Matem\'aticas, Facultad de Ciencia y Tecnologia, \\
Universidad del Pais Vasco, E-48080 Bilbao, Spain \\
\textit{\&}  Basque Center of Applied Mathematics, (BCAM)\\
Alameda de Mazarredo 14, E--48009 Bilbao, Spain.\\
email: mtpesmam@ehu.es\\

\end{document}